\documentclass[10pt]{article}
\usepackage{amsmath}
\usepackage{amsbsy}
\usepackage{amssymb}
\usepackage{amsthm}
\usepackage{anyfontsize}
\usepackage{epsfig}
\usepackage{epstopdf}
\usepackage{wrapfig}
\usepackage{color,cancel}
\usepackage{fullpage}
\usepackage{float}
\usepackage[demo]{adjustbox}
\usepackage[demo]{graphicx}
\usepackage{caption}
\usepackage{subcaption}
\usepackage{bbm}
\usepackage{calrsfs}
\usepackage{url} 
\pdfoutput=1

\DeclareMathOperator*{\esssup}{ess\,sup}

\usepackage{amsmath,amsfonts,amssymb,color,graphicx}
\newtheorem{theorem}{Theorem}[section]
\newtheorem{Algorithm}{Algorithm}[section]
\newtheorem{corollary}{Corollary}[section]
\newtheorem{lemma}{Lemma}[section]

\newtheorem{remark}{Remark}[section]

\numberwithin{equation}{section}
\newcommand{\norm}[1]{\left\Vert#1\right\Vert}

\newcommand{\bu}{{\bf u}}
\newcommand{\bv}{{\bf v}}
\newcommand{\bw}{{\bf w}}
\newcommand{\bfx}{{\bf x}}
\newcommand{\bx}{{\bf x}}
\newcommand{\be}{{\bf e}}
\newcommand{\bff}{{\bf f}}

\newcommand{\bphi}{{\boldsymbol \phi}}

\newcommand{\bfX}{{\bf X}}
\newcommand{\bfV}{{\bf V}}

\usepackage{color}
\usepackage[section]{placeins}
\makeatletter
\AtBeginDocument{%
	\expandafter\renewcommand\expandafter\subsection\expandafter{%
		\expandafter\@fb@secFB\subsection
	}%
}

\date{}

\title{Time Filtered Second Order Backward Euler Method for EMAC Formulation of Navier-Stokes Equations  }

\author{
	Medine Demir
	\thanks{Department of Mathematics, Middle East Technical University, 06800 Ankara, Turkey; dmedine@metu.edu.tr}
	\and
	Aytekin  \c{C}{\i}b{\i}k
	\thanks{Department of Mathematics, Gazi University, 06560 Ankara, Turkey; abayram@gazi.edu.tr}
	\and
	Song\"{u}l Kaya
	\thanks{
		Department of Mathematics, Middle East Technical University, 06800 Ankara, Turkey; smerdan@metu.edu.tr.}
}
\begin{document}

\maketitle
\textbf{Abstract.} This paper considers the backward Euler based linear time filtering method for the developed energy-momentum-angular momentum conserving (EMAC) formulation of the time dependent-incompressible Navier-Stokes equations in the case of weakly enforced divergence constraint. The method adds time filtering as a post-processing step to the EMAC formulation to enhance the accuracy and to improve the approximate solutions. We show that  in comparison with the Backward-Euler based EMAC formulation without any filter, the proposed method not only leads to a 2-step, unconditionally stable and second order accurate method but also increases numerical accuracy of solutions. Numerical studies verify the theoretical findings and demonstrate preeminence of the proposed method over the unfiltered case. 


\section{Introduction}

 Due to its widespread use in both science and engineering, the importance of Navier-Stokes equations (NSE) in computational fluid dynamics is indisputable. Since finding an analytical solution is extremely hard, especially under small viscosity, many studies have been done over decades related with Galerkin finite element methods for incompressible flows, see, e.g., \cite{ VP, W, HMJ, K}. The success of numerical simulation of NSE's depends on  physically motivated discretization schemes. In classical $H$-conforming methods, the divergence constraint is only weakly enforced which leads to loss of numerical accuracy as well as many important conservation laws, including energy, momentum, angular momentum, and others. It is well-known that a good way to measure the accuracy of a model is by how much physical balances it retains. For this purpose, practitioners have developed many numerical methods, see e.g. \cite{ JW, ML, L07}.  However, none of them conserve all the balances of physical quantities such as  energy-momentum-angular momentum conserving (EMAC) at the same time. To handle this issue, in the study \cite{OTML},  the authors introduce a new formulation of NSE by reformulating the nonlinear term, named as the EMAC formulation which is the first scheme conserving all of these mentioned quantities even under the weakly enforced divergence constraint. In the light of studies following the original paper \cite{OTML}, e.g. \cite{STML, DO, OG}, the EMAC formulation has proven to exhibit superior performance  especially over longer time periods \cite{ML20} compared to traditionally used schemes based on skew-symmetric formulations. All these studies agree in the opinion that the EMAC formulation reduces numerical dissipation and significantly increases numerical accuracy compared to previously studied ones  thanks to its ability to conserve all the physical quantities.
 
The central challenge in the CFD community is time accuracy to reflect important physical features of      solutions. From the practical point of view, incorporating minimum complexity into existing codes and increasing numerical accuracy are critical for many purposes. In general, Backward Euler (BE) time discretization is well-known to be one of the most commonly used method to approximate the time dependent viscous flow problems, \cite{VG,FJ}. BE is the basis for construction of more complex methods due to its stability, rapid convergence to analytical solution and easy implementation. However,  it causes spurious  oscillations in flow physics \cite{GS}  for larger time step sizes even if the method is stable. The way to increase numerical accuracy and enhance physical features for fully implicit and BE methods was first proposed in the study \cite{GW} for the ODEs. 

As stated in \cite{GS}, by adding only one additional line of code to BE scheme, it is possible to reduce numerical dissipation, to increase accuracy from first order to the second order and to obtain an $A$-stable method with a useful error estimator. Beside constant time steps, the methods could be extended also to variable time step sizes. The method, called time filtering, is extended to the incompressible NSE in \cite{DZ}, to MHD equations in \cite{AC} and to the Boussinesq equations in \cite{MN}.  The common theme in these studies is that applying the linear time filtering method to BE scheme with constant time step achieves a second order, more accurate and an $A$-stable method. Moreover, combining the second step with the first step yields a second order accurate method akin to backward differentiation (BDF2) formula and performs a consistent and simple stability/convergence analysis. While a family of general variable step size time filter algorithms is investigated by \cite{DS},  general linear methods by using pre-filter or post-filter by using constant step size have been considered in \cite{DL}. Such filtering processes yields embedded higher order methods with minimum complexity, see \cite{DL}.

The purpose of this report is to investigate the effect of this novel idea from \cite{DZ} on the EMAC formulation of time-dependent, incompressible fluid flows for constant time steps. Herein, based on the success of EMAC formulation in \cite{OTML} and time filtering on BE for constant time step in \cite{DZ}, we naturally took the step of combining these two ideas in order to see whether time filtering method will improve the accuracy of solutions of NSE with EMAC formulation. The proposed numerical scheme is a two-step time filtered BE method which is efficient, $O(\Delta t^2)$, $A$-stable and easy to adapt into any existing code. In the first step, velocity solutions of the EMAC scheme is calculated with the usual BE finite element discretization, which we call as BE-EMAC. The second step post proceeds this velocity approximation by using a second order time filtering under constant step size. Thus, combining the BE discretized EMAC formulation of NSE with time filtering, one benefits from the simplicity of BE method and does not suffer from its mentioned drawbacks. The EMAC treatment of non-linearity uses Newton method. Although the BE time stepping scheme does not conserve energy exactly as Crank-Nicolson scheme does, our aim here is to show the improvement which has been done due to the addition of time filter post-processing. We do not assert the superiority of the scheme over Crank-Nicolson based time stepping schemes in terms of conservation of total kinetic energy.
We show the conservation of the physically conserved quantities such as (a modified) energy, momentum and angular momentum and we refer it as energy, momentum, angular momentum conserving time filtered formulation (EMAC-FILTERED). Additionally, we provide that the method is both stable and optimally accurate. To the best of authors' knowledge, this is the first attempt to combine EMAC formulation with a time filtering post-processing for solving NSE numerically.

The presentation of this paper is as follows. Section 2 provides some notations and mathematical preliminaries needed for a smooth analysis to follow. In Section 3, EMAC-FILTERED  method is described. Conservation properties are studied in Section 4.  Section 5 is devoted to a complete stability and convergence  analysis of the EMAC-FILTERED method. In Section 6, several numerical experiments are performed which test the conservation properties, the accuracy and the efficiency of the method and compares it with BE-EMAC solutions. Finally, conclusions and possible future directions are drawn in Section 7.

\section{ Notations and Mathematical Preliminaries}

 Let $\Omega$ in ${\rm I\!R} ^d$, $ (d=2,3) $ is a convex polygonal or polyhedral domain  with boundary $\partial \Omega$. In this work, we consider NSE 
\begin{equation}\label{nse}
	\begin {array}{rcll}
	\bu_t -\nu\triangle \bu+\bu\cdot \nabla \bu+ \nabla p &=& \bf f &
	\mathrm{in }\ \Omega\times (0,T], \\
	\nabla \cdot \bu&=& 0& \mathrm{in }\ \Omega\times  (0,T],\\
	\bu&=& \mathbf{0}& \mathrm{on }\ \partial\Omega\times[0,T],\\
	\bu(\bx,0)&=& \bu_{0}(\bx) & \mathrm{for}\  \bx \in \Omega.
\end{array}
\end{equation}
Here $\bu$ stands for  the velocity, $p$ the zero-mean pressure, $\bff$ an external force and $\nu$ the kinematic viscosity.

We  use the following notations throughout the paper. Let Lebesgue spaces $(L^{p}(\Omega))^{d}$, $ 1\leq p\leq \infty $, $ p\neq 2$ be equipped with  $(\cdot,\cdot)_{L^p}$ and the  norm $\norm{.}_{L^{p}}$, respectively. In particular, $L^2$-inner product and $L^2$- norm are denoted by $(\cdot,\cdot)$ and $\norm{\cdot}$, respectively. The norm in Sobolev spaces $(H^{k}(\Omega))^{d}$ is denoted by  $\norm{.}_{k}$ and all other norms will be labeled appropriately. The natural velocity field and pressure spaces are
 \begin{eqnarray}
 &&\bfX:=(H_0^1(\Omega))^{d}:=\{v\in L^2(\Omega)^{d}\ : \nabla v \in L^2(\Omega)^{d\times d}\ \text{and} \ v=0 \ \text{on} \ \partial\Omega \}, \\
 &&Q:=L_0^2(\Omega):=\{q\in L^2(\Omega): \int_{\Omega}q \,d\bx=0 \}.
 \end{eqnarray}
 The dual space of $\bfX$ is denoted by $H^{-1}(\Omega)$ and is endowed with norm $\norm{ \cdot}_{-1} $.
 
 The divergence free subspace $\bfV$ of $\bfX$ is defined by
 $$\bfV := \{\bv \in \bfX: (\nabla\cdot \bv, q)=0, \quad \forall  q \in  Q\}.$$
For all Lebesgue measurable functions $w: [0,T]\rightarrow \bfX $, the following norms are denoted by
 \begin{eqnarray}
 \norm{w}_{L^p(0,T;X)}&=&\bigg(\int_0^T\norm{w(t)}^p_{\bfX}dt\bigg)^{1/p}, \quad 1\leq p < \infty \nonumber\\
 \norm{w}_{L^{\infty}(0,T;X)}&=&\esssup_{0\leq t\leq T}\norm{w(t)}_{\bfX} \nonumber
\end{eqnarray}
For a spatial discretization, let $\bfX_h \subset \bfX, Q_h\subset Q$ be two families of finite element spaces defined on a regular, conforming triangulation $\pi_h$  of the domain $\Omega$ with maximum diameter $h$ satisfying the discrete inf-sup condition,
  \begin{eqnarray}
  \inf_{q_h\in{Q}^h}\sup_{\bv_h\in {\bfX}^h}\frac{(q_h,\,\nabla\cdot
  	\bv_h)}{||\,\nabla \bv_h\,||\,||\,q_h\,||}\geq \beta_0 > 0.
  \label{infsup}
  \end{eqnarray}
   Define the discretely divergence-free subspace of $\bfX_h$ by
  \begin{eqnarray}
  \bfV_{h}&:=&\{ \bv_{h}\in \bfX_{h} \ | (\nabla\cdot \bv_{h},q_{h})=0 \ \forall q_{h} \  \in Q_{h}\}.\nonumber
  \end{eqnarray}
Note that, under the inf-sup condition, (\ref{infsup}) the variational formulation of NSE (\ref{nse})
in $(\bfX_h,Q_h)$ is equivalent to in $(\bfV_h,Q_h)$, see, e.g.,  \cite{GR79}. 


For any $v \in \bfX_h$, the standard inverse inequality is given by 
 \begin{eqnarray}
 \norm{\nabla\bv}\leq Ch^{-1}\norm{\bv} \quad \forall \bv \in H_0^1(\Omega)\label{inv}.
 \end{eqnarray}

In addition, it will be assumed that the spaces $(\bfX_h,Q_h)$ comprise piecewise polynomials of degree at most $(k,k-1)$ and the following well-known approximation properties hold for all $(v_h, q_h) \in ({\bf X}_h,Q_h)$, (see, e.g., \cite{Joh16}),
\begin{eqnarray}
\inf_{\bv_h\in{\bfX}_h} \left( \|(\bu-\bv_h)\|+h\|\nabla(\bu-\bv_h)\| \right)&\leq&
C h^{k+1} \| \bu \|_{k+1}\quad{\bu} \in H^{k+1}(\Omega),\label{ap1}
\\
\inf_{q_h\in{Q}_h}\|p-q_h\|&\leq & Ch^k \|p\|_{k}\quad {p} \in H^{k}(\Omega).\label{ap10}
\end{eqnarray}
For the convergence analysis, we also need the following discrete Gronwall's lemma stated in \cite{GR79} :
\begin{lemma}(Discrete Gronwall Lemma)\label{gl}
	Let $\Delta t$, M and $\alpha_{n},\beta_{n},\xi_{n},\delta_{n}$ (for integers n$\geq$ 0) be nonnegative finite numbers such that
	\begin{eqnarray}
	\alpha_{m}+\Delta t\sum_{n=0}^{m} \beta_{n} \leq \Delta t\sum_{n=0}^{m}\delta_{n}\alpha_{n}+\Delta t\sum_{n=0}^{m}\xi_{n}
	+M
	\end{eqnarray}
	Suppose that $\Delta t \delta_{n}<1$ for all n. Then,
	\begin{eqnarray}
	\alpha_{m}+\Delta t\sum_{n=0}^{m} \beta_{n} \leq \exp \Big(\Delta t\sum_{n=0}^{m}\beta_n\dfrac{\delta_{n}}{1-\Delta t \delta_{n}}\Big)\Big(\Delta t\sum_{n=0}^{m}\xi_{n}
	+M\Big) \quad \text{for} \quad m\geq 0
	\end{eqnarray}
	
\end{lemma}
\subsection{Preliminaries for EMAC formulation}
We now present some vector identities and properties of the trilinear forms which will have a crucial role in introducing the scheme that we propose.
For a sufficiently smooth velocity field $\bu$, let us denote the symmetric part of $\nabla \bu$ by $\bf \mathbb{D}(\bu)=\frac{\nabla\bu+(\nabla\bu)^T}{2}$ which is the velocity deformation tensor. The following identity is the key idea for the EMAC formulation
\begin{equation}
(\bu\cdot\nabla)\bu=2\bf \mathbb{D}(\bu)\bu-\frac{1}{2}\nabla|\bu|^2
\end{equation}
for which the inertia term is splitted into the acceleration driven by $2\mathbb{D}(\bu)$ and the  potential term further absorbed by redefined pressure.
Following \cite{DFM}, the trilinear form for EMAC formulation is defined by
\begin{equation}
c(\bu,\bv,\bw)=2(\bf \mathbb{D}(\bu)\bv,\bw)+((\nabla\cdot\bu)\bv,\bw)
\end{equation}
For purpose, it is important to assume that $\bu,\bv,\bw \in \bfX$, i.e. no divergence free condition is assumed for any of $\bu, \bv, \bw$. It is critical to add the divergence term in the definition of $c(\cdot,\cdot,\cdot)$ to satisfy the cancellation property, $c(\bv,\bv,\bv)=0$.
The following identities are also used:
\begin{eqnarray}
(\bu\cdot\nabla\bv,\bw)&=&-(\bu\cdot\nabla\bw,\bv)-((\nabla\cdot\bu)\bv,\bw),\label{i1}\\
(\bu\cdot\nabla\bw,\bw)&=&-\frac{1}{2}((\nabla\cdot\bu)\bw,\bw),\label{i2}\\
(\bu\cdot\nabla\bv,\bw)&=&((\nabla\bv)\bu,\bw)=((\nabla\bv)^T\bw,\bu)\label{i3}.
\end{eqnarray}

\subsection{Preliminaries for Time Filtering}

For preliminaries related time filtering, we first start with $G$-norm. Since $G$-stability refers to $A$-stability, the $G$-matrix is often used in BDF2 analysis. The authors refer to \cite{EG02} and references therein about $G$-stability. With respect to the definition of $G$-matrix in \cite{JMR}, with the choices of $\theta=1$ and $\nu=2\epsilon$, $G$-matrix of the described method is as follows:
$$G=\begin{pmatrix}
	\frac{3}{2}&-\frac{3}{4}\\-\frac{3}{4}&\frac{1}{2}
\end{pmatrix}$$
with the $G$-norm
\begin{eqnarray}
	\norm{ \begin{bmatrix}
			\bu\\
			\bv
	\end{bmatrix}}_{G}^{2} =\bigg( \begin{bmatrix}
		\bu\\
		\bv
	\end{bmatrix},G \begin{bmatrix}
		\bu\\
		\bv
	\end{bmatrix}\bigg). \label{Gnorm}
\end{eqnarray}
Here $\begin{bmatrix}
	\bu\\
	\bv
\end{bmatrix} $ is a $2n$ vector.

The proposed algorithm requires also the definition of $F$-norm. We use a symmetric positive matrix  $F=3I\in\mathbb{R}^{n\times n}$ which is defined in details in \cite{JMR}. For any $\bu\in \mathbb{R}^{n}$, define $F$ norm of the $n$ vector $\bu$ by
\begin{eqnarray}
	\norm{\bu}_{F}=(\bu,F\bu) \label{Fnorm}
\end{eqnarray}
The following properties of $G$-norm are well known and see, \cite{EG02, JMR} for a detailed derivation of these estimations.
\begin{lemma}
	$L^2$-norm and $G$-norm are equivalent in the following sense: there exists constants $C_1>C_2>0$ such that
	$$ C_1 \norm{ \begin{bmatrix}
			\bu\\
			\bv
	\end{bmatrix}}_{G}^{2}\leq  \norm{ \begin{bmatrix}
			\bu\\
			\bv
	\end{bmatrix}}^{2} \leq C_2 \norm{ \begin{bmatrix}
			\bu\\
			\bv
	\end{bmatrix}}_{G}^{2}.$$
\end{lemma}
\begin{lemma}\label{lem:inn}
	The symmetric positive matrix  $F$ $\in$  ${\rm I\!R}^{n\times n}$ and the symmetric matrix $G$ $\in$  ${\rm I\!R}^{2n\times 2n}$ satisfy the following equality:
	
	\begin{eqnarray}
		\Big(\frac{\frac{3}{2}w_{n+1}-2 w_n+\frac {1}{2}w_{n-1}}{\Delta t} ,   \frac{3}{2} w_{n+1}- w_n+ \frac{1}{2} w_{n-1}\Big) \nonumber
		\\
		=\frac{1}{\Delta t}\norm{ \begin{bmatrix}
				w_{n+1} \\
				w_{n}
		\end{bmatrix}} _{G}^{2} -\frac{1}{\Delta t}\norm { \begin{bmatrix}
				w_{n} \\
				w_{n-1}
		\end{bmatrix}}_{G}^{2} 	+\dfrac{1}{4\Delta t}\norm{w_{n+1}-2w_{n}+w_{n-1}}_{F}^{2} \label{innpro}
	\end{eqnarray}
	
\end{lemma}

\begin{lemma}\label{lem:gnorm}
	For any $\bu$,$\bv$ $\in$ $R^{n}$, we have
	\begin{eqnarray} 	
		\big( \begin{bmatrix}
			\bu\\
			\bv
		\end{bmatrix},G \begin{bmatrix}
			\bu\\
			\bv
		\end{bmatrix}\big)
		\geq\dfrac{3}{4}\norm{\bu}^2-\dfrac{1}{4}\norm{\bv}^2 ,\label{first}
	\end{eqnarray}
	\begin{eqnarray} 	
		\big( \begin{bmatrix}
			\bu\\
			\bv
		\end{bmatrix},G \begin{bmatrix}
			\bu\\
			\bv
		\end{bmatrix}\big)
		\leq\dfrac{3}{2}\norm{\bu}^2+\dfrac{3}{4}\norm{\bv}^2 .\label{second}
	\end{eqnarray}
\end{lemma}
\begin{proof}
	Taking  $\theta=1$ and $\nu=2\epsilon$ in Lemma 3.1 of \cite{JMR} gives the result.
\end{proof}
In the analysis, we also need the following consistency error estimations.
\begin{lemma}
	There exists $C>0$ such that
	\begin{eqnarray}
		\Delta t\sum_{n=1}^{N-1}\norm{F[\bw^{n+1}]-\bw^{n+1}}^2&\leq& C\Delta t^4\norm{\bw_{tt}}_{L^2(0,T;L^2(\Omega))}\\
		\Delta t\sum_{n=1}^{N-1}\norm{	\frac{3w_{n+1}-4 w_n+w_{n-1}}{2\Delta t}- \bw^{n+1}}^2 &\leq&C\Delta t^4\norm{\bw_{ttt}}_{L^2(0,T;L^2(\Omega))}
	\end{eqnarray}
	\begin{proof}
		The reader is referred to \cite{AC} for a detailed proof.
	\end{proof}
\end{lemma}

\section{ EMAC-FILTERED scheme }
We consider the following weak formulation  of the EMAC formulation of (\ref{nse}): Find $\bu: (0,T] \longrightarrow \bfX$, $P: (0,T] \longrightarrow Q$ satisfying
\begin{eqnarray}
(\bu_t,\bv)+\nu(\nabla\bu,\nabla\bv)+c(\bu,\bu,\bv)-(P,\nabla\cdot\bv)&=&(\bff,\bv)\quad \forall \bv \in \bfX \label{vf}\\
(q,\nabla\cdot\bu)&=&0 \quad \forall q \in Q, \label{vf1}
\end{eqnarray}
where $P$ is defined as $P=p-\dfrac{1}{2}|\bu|^2$ and $\bu(0,\bf x)=\bu_0(\bf x) $ $\in \bfX$. Here, we use the EMAC form of the nonlinear term. We now present the two step time filtered numerical scheme of (\ref{vf})-(\ref{vf1}) for constant time step. In the first step, the velocity approximation of the scheme (\ref{vf})-(\ref{vf1}) is calculated with the usual BE finite element (fully implicit) discretization and the second step introduces a simple time filter which combines this velocity solution linearly with the solutions at previous time levels. The second step increases time accuracy remarkably although it does not significantly alter system complexity. The modular time filtered numerical scheme of (\ref{vf})-(\ref{vf1}) of the NSE with EMAC formulation reads as follows:

\begin{Algorithm}\label{alg}
Let body force $f$ and the initial condition $\bu_0$ be given. Select T as the end time, and let N be the number of time steps to take the time step size   $\Delta t=T/N$. Define $\bu_h^0$, $\bu_h^{-1}$ as the nodal interpolants of $\bu_0$, then for any $n\geq 1$  ($n=0,1,,...,N-1$), find
	$(\bu_{h}^{n+1},{P}_{h}^{n+1}) \in (\bfX_{h},Q_{h})$ via the following two steps:
	\\
		{	\bf Step 1:} Compute $(\tilde{\bu}_{h}^{n+1},P_h^{n+1}) \in  (\bfX_{h},Q_{h})$ such that
	
	\begin{eqnarray}
	\bigg(\dfrac{\tilde{\bu}_{h}^{n+1}-\bu_{h}^{n}
		}{\Delta t},\bv_{h}\bigg)+\nu(\nabla \tilde{\bu}_{h}^{n+1},\nabla \bv_{h})+c(\tilde{\bu}_{h}^{n+1},\tilde{\bu}_{h}^{n+1},\bv_h)-(P_{h}^{n+1},\nabla \cdot \bv_{h})
&=&(\bff(t^{n+1}),\bv_{h})\label{a1},
	\\
	(\nabla \cdot \tilde{\bu}_{h}^{n+1},q_{h})&=&0. \label{a2}
	\end{eqnarray}
	\bf{Step 2:}
	\begin{eqnarray}
	\bu_h^{n+1}=\tilde{\bu}_{h}^{n+1}-\frac{1}{3}(\tilde{\bu}_{h}^{n+1}-2\bu_{h}^{n}+\bu_{h}^{n-1})\label{a3}
	\end{eqnarray}
\end{Algorithm}
for all $ (\bv_{h},q_{h}) \in (\bfX_{h},Q_{h})$.
\begin{remark} 
	One option in Algorithm 3.1 is that also filtering pressure as $P_h^{n+1}=\tilde{P}_{h}^{n+1}-\frac{1}{3}(\tilde{P}_{h}^{n+1}-2P_{h}^{n}+P_{h}^{n-1})$. However, as noted in \cite{AC,DZ}, not filtering pressure yields better numerical results because of consistency terms arising from pressure equations. Based on these previous experiences in \cite{AC,DZ}, we choose not to filter pressure.  
\end{remark}
We note that, Step 1 is the standard backward Euler scheme for the NSE. Step 2 is  just a post processing step and it can be easily applied into existing codes. Clearly, this implemented linear time filter makes the method numerically 
efficient. Define the interpolation operator $F$ by
\begin{eqnarray}
F[\bw_h^{n+1}]=\frac{3}{2}\bw_h^{n+1}-\bw_h^n+\frac{1}{2}\bw_h^{n-1}, \label{F1}
\end{eqnarray}
which is formally $F[\bw_h^{n+1}]=\bw_h^{n+1}+ O(\Delta t^2) $.
Rewriting (\ref{a3}) as $\tilde{\bu}_{h}^{n+1}=\frac{3}{2}\bu_{h}^{n+1}-\bu_{h}^{n}+\frac{1}{2}\bu_{h}^{n-1}$ and inserting in (\ref{a1})-(\ref{a2}) along with (\ref{F1}), the following equivalent method is obtained.
\begin{eqnarray}
\frac{1}{\Delta t}(\frac{3}{2}\bu_h^{n+1}-2\bu_h^n+\frac{1}{2}\bu_h^{n-1},\bv_h)+\nu(\nabla(F[\bu_h^{n+1}]),\nabla\bv_h)+c(F[\bu_h^{n+1}],F[\bu_h^{n+1}],\bv_h)\nonumber\\
-(P_h^{n+1},\nabla\cdot\bv_h)=({\bff}(t^{n+1}),\bv_h) \label{E}
\\
(\nabla\cdot(F[\bu_h^{n+1}]),q_h)=0 \label{D}
\end{eqnarray}
for all $(\bv_h,q_h)\in (\bfX_h,Q_h)$. We note that to simplify the stability and convergence analysis, the equivalent formulation  (\ref{E})-(\ref{D}) will be used. However, the formulation (\ref{a1})-(\ref{a3}) will be considered in the implementation of the method for computer simulations.

\begin{remark}
 We emphasize here that although the time derivative is discretized by using classical BDF2 time stepping method, the other terms are not. Thus, the method should not be considered as the standard BDF2 method. We refer to \cite{DZ}, for details of the time filtering approach.
\end{remark}

\section{Conservation Laws}\label{seccon}
In this section, we investigate the conservation of the integral variants of fluid flow energy, momentum and angular momentum for Algorithm \ref{alg}. It is well known that the physical accuracy of a model is measured by how well its solutions preserve these quantities. For NSE, energy, momentum and angular momentum are defined by
$$ \text{Kinetic Energy}: \quad E=\frac{1}{2}(\bu,\bu)=\frac{1}{2}\int_{\Omega}|\bu|^2,d{\bfx},$$
$$ \text{Linear Momentum}: \quad M=\int_{\Omega}\bu \,d{\bfx},$$
$$ \text{Angular Momentum}: \quad AM=\int_{\Omega}\bu \times {\bfx}\,d{\bfx}.$$
Let $e_i$ be the $i$-th unit vector and $\phi_i=x \times e_i$. Then, momentum and angular momentum can be equivalently defined by
$$M_i=\int_{\Omega}\bu_i\, d{\bfx}=(\bu,e_i),$$
$$ (AM)_i=\int_{\Omega}(\bu \times x)\cdot e_i  \, d{\bfx}=(\bu,\phi_i).$$
It is important to assume that the solutions $\bu, p$ have compact support in $\Omega$ (e.g. consider an isolated vortex). We first state the energy balance of the EMAC-FILTERED scheme (\ref{E})-(\ref{D}).
\begin{theorem}
	 EMAC-FILTERED scheme (\ref{E})-(\ref{D}) conserves a modified kinetic energy  for $\nu=0, \bff=0$:
	\begin{eqnarray}
	\norm{ \begin{bmatrix}
		\bu^{N}_h \\
		\bu^{N-1}_h
		\end{bmatrix}} _{G}^{2} +\Delta t\sum_{n=1}^{N-1}(\nu\norm {\nabla F[ \bu_h^{n+1}]}^2)+
	\frac{1}{4}\sum_{n=1}^{N-1}\norm{\bu_h^{n+1}-2\bu_h^n+\bu_h^{n-1}}^2_F\nonumber\\
	=\norm{ \begin{bmatrix}
		\bu^{1}_h \\
		\bu^{0}_h
		\end{bmatrix}} _{G}^{2} +\Delta t\sum_{n=1}^{N-1}(\bff(t^{n+1}),F[\bu_h^{n+1}]) \label{engy}
	\end{eqnarray}
	\end{theorem}

\begin{remark}
  We should point out here that, one can obtain an exact energy conservation by making use of Crank-Nicolson time discretization along with an EMAC scheme. The aim here is to show the improvement of classical BE based scheme by adding simple time filter post processing and obtain second order accuracy in terms of time. Also note that the numerical dissipation $\frac{3}{4}\norm{\bu_h^{n+1}-2\bu_h^n+\bu_h^{n-1}}^2\approx \frac{3}{4}\Delta t^4\norm{\bu_{tt}(t^{n+1})}^2$, of the scheme is asymptotically smaller than the numerical dissipation of Backward Euler method.
	\end{remark}

	\begin{proof}
	Set $\bv_h=	F[\bu_h^{n+1}]$ in (\ref{E}), $q_h=p_h^{n+1}$ in (\ref{D}), then the nonlinear term and pressure term vanishes. Then with the use of  Lemma \ref{lem:inn}, one gets
		\begin{eqnarray}
	\frac{1}{\Delta t}\norm{ \begin{bmatrix}
		\bu^{n+1}_h \\
		\bu^{n}_h
		\end{bmatrix}} _{G}^{2} -\frac{1}{\Delta t}\norm{ \begin{bmatrix}
		\bu^{n}_h \\
		\bu^{n-1}_h
		\end{bmatrix}} _{G}^{2}+
	\frac{1}{4\Delta t}\norm{\bu_h^{n+1}-2\bu_h^n+\bu_h^{n-1}}^2_F+\nu\norm {\nabla F[ \bu_h^{n+1}]}^2
	=(\bff(t^{n+1}),F[\bu_h^{n+1}])\label{dt} .
	\end{eqnarray}
	Multiplying both sides of (\ref{dt}) by $\Delta t$ and taking sum from $n=1$ to $N-1$ proves the result.
	\end{proof}
Next, we consider the conservation of momentum and angular momentum of  Algorithm \ref{alg}. Let $\Omega_s$ be a strip around $\partial\Omega$ and  $\Omega_i$ be such that $\Omega=\Omega_s\cup \Omega_i $. Let us also define $\chi(g) \in \bfX$ to be the restriction of some arbitrary function $g$ by setting $\chi(g)=g $ in $\Omega$ and arbitrarily defined elsewhere to meet the boundary conditions such that $\chi(g)=g$ in $\Omega_i$ but $g|_{\partial\Omega}=0$. Based on \cite{OTML}, we assume that $\bu_h=0$ and $p_h=0$ on $ \Omega_s$.
\begin{theorem}
	For $\nu=0, \bff=0$, EMAC-FILTERED scheme (\ref{E})-(\ref{D}) conserves momentum and angular momentum for all $t>0$, i.e.,
$$M_{\text{Emac-Fil}}(t)=M_{\text{Emac-Fil}}(0)$$
$$AM_{\text{Emac-Fil}}(t)=AM_{\text{Emac-Fil}}(0)$$

\end{theorem}
\begin{proof}
	We start by showing momentum conservation. Choose $\bv_h=\chi(e_i)$ in (\ref{E}) to get
	\begin{equation}
	((\bu_h)_t ,   e_i) +\nu(\nabla F[\bu_h^{n+1}],\nabla e_i)+c( F[ \bu_h^{n+1}], F[ \bu_h^{n+1}],e_i)=(\bff(t^{n+1}),e_i). \label{eee1}
	\end{equation}
	For the nonlinear term in (\ref{eee1}), expand the rate of deformation tensor and use the identity (\ref{i1}) along with the fact that $e_i$ is divergence-free. This yields
	\begin{eqnarray}
	c( F[ \bu_h^{n+1}], F[ \bu_h^{n+1}],e_i)&=&2({\bf D}(F[ \bu_h^{n+1}])F[ \bu_h^{n+1}],e_i)+(div(F[ \bu_h^{n+1}])F[ \bu_h^{n+1}],e_i)\nonumber\\
&=&	(F[ \bu_h^{n+1}]\cdot\nabla F[ \bu_h^{n+1}],e_i)+(e_i\cdot \nabla F[ \bu_h^{n+1}], F[ \bu_h^{n+1}] )+((\nabla\cdot F[ \bu_h^{n+1}])F[ \bu_h^{n+1}],e_i)\nonumber\\
&=& ( e_i\cdot\nabla F[ \bu_h^{n+1}],F[ \bu_h^{n+1}] )\nonumber\\
&=&0.\nonumber
	\end{eqnarray}
	Under the assumption $\nu=0, \bff=0$, one gets
	$$\frac{d}{dt}(\bu_h,e_i)=0$$
	 which is precisely the conservation of momentum.
	
For conservation of angular momentum, take $\bv_h=\chi(\phi_i)$ in (\ref{vf}) to get
 	\begin{eqnarray}
 ((\bu_h)_t ,   \phi_i) +\nu(\nabla F[\bu_h^{n+1}],\nabla \phi_i)+c( F[ \bu_h^{n+1}], F[ \bu_h^{n+1}],\phi_i)=(\bff(t^{n+1}),\phi_i)\nonumber
 \end{eqnarray}
In a similar manner, by using the identities (\ref{i1}) and (\ref{i2}) respectively along with $\nabla\cdot \phi_i=0$, the nonlinear term reduces to
	\begin{eqnarray}
c( F[ \bu_h^{n+1}], F[ \bu_h^{n+1}],\phi_i)&=&2({\bf D}(F[ \bu_h^{n+1}])F[ \bu_h^{n+1}],\phi_i)+(div(F[ \bu_h^{n+1}])F[ \bu_h^{n+1}],\phi_i)\nonumber\\
&=&	(F[ \bu_h^{n+1}]\cdot\nabla F[ \bu_h^{n+1}],\phi_i)+(F[ \bu_h^{n+1}], \nabla F[ \bu_h^{n+1}], \phi_i)+((\nabla\cdot F[ \bu_h^{n+1}])F[ \bu_h^{n+1}],\phi_i)\nonumber\\
&=& (F[ \bu_h^{n+1}]\cdot\nabla F[ \bu_h^{n+1}],\phi_i)+((\nabla F[ \bu_h^{n+1}])F[ \bu_h^{n+1}],\phi_i)\nonumber\\
&=&-(F[ \bu_h^{n+1}]\cdot\nabla \phi_i,F[ \bu_h^{n+1}])\nonumber
\end{eqnarray}
Note that, expansion of the last term gives $(F[ \bu_h^{n+1}]\cdot\nabla \phi_i,F[ \bu_h^{n+1}])=0$, i.e., the non-linear term vanishes. The use of $\nu=0, \bff=0$ results in
$$\frac{d}{dt}(\bu_h,\phi_i)=0,$$
which is the required statement of angular momentum conservation.
\end{proof}
\section{Numerical Analysis}
This section provides unconditional stability result and convergence analysis of the proposed Algorithm \ref{alg}.  We first provide the stability analysis of the method.
\begin{lemma}
	 Let $ \bff$ $\in L^{\infty}(0,T;H^{-1}(\Omega)) $. Then for all $\Delta t>0$ and $N\geq 1$, the solution of Algorithm \ref{alg} is unconditionally stable in the following sense:
	 \begin{eqnarray}
	 \lefteqn{\norm {\bu^{N}_{h}}^{2} + \dfrac{1}{3}\sum_{n=1}^{N-1} \norm {\bu^{n+1}_{h}-2\mathbf \bu^{n}_{h}+\mathbf \bu^{n-1}_{h}}_{F}^{2}+\dfrac{2\Delta t\nu }{3}\sum_{n=1}^{N-1}\norm {\nabla  F[\bu_{h}^{n+1}]}^{2}} \nonumber
	 \\
	 &\leq& \bigg(\dfrac{1}{3}\bigg)^N
	 \norm{
	 	\bu^{0}_{h}}^{2} +2N(\norm{\mathbf \bu^{1}_h}^{2}+\norm{\mathbf \bu^{0}_h}^{2})+\dfrac{2N\Delta t \nu^{-1}}{3}\sum_{n=1}^{N-1}\norm {\bff(t^{n+1})}^{2}.
	 \label{sta}
	 \end{eqnarray}
\end{lemma}
\begin{proof}
	We start the proof by the global energy conservation (\ref{engy}). 
The application of Cauchy-Schwarz inequality, Young's inequality and the dual norm on the forcing term gives
\begin{eqnarray}
(\bff(t^{n+1}),F[\bu_h^{n+1}])\leq \frac{\nu^{-1}\Delta t}{2}\norm{\bff(t^{n+1})}^2_{-1}+\frac{\nu\Delta t}{2}\norm{\nabla F[\bu_h^{n+1}]}^2.
\end{eqnarray}
Inserting the estimate in  (\ref{engy}) and applying Lemma \ref{lem:gnorm} leads to
 \begin{eqnarray}
 \frac{3}{4}\norm{\bu_h^N}^2+
 \frac{1}{4}\sum_{n=1}^{N-1}\norm{\bu_h^{n+1}-2\bu_h^n+\bu_h^{n1}}^2_F
  +\frac{\nu\Delta t}{2}\sum_{n=1}^{N-1}(\norm {\nabla F[ \bu_h^{n+1}]}^2)\nonumber\\
  \leq  \frac{1}{4}\norm{\bu_h^{N-1}}^2+ \frac{3}{2}\norm{\bu_h^1}^2 +\frac{3}{4}\norm{\bu_h^0}^2+ \frac{\nu^{-1}\Delta t}{2}\norm{\bff(t^{n+1})}^2_{-1}
 \end{eqnarray}
Lastly, the proof is completed by multiplying by $ \frac{4}{3}$ and using the induction.
\end{proof}
We proceed to present a detailed convergence analysis of the proposed time filtered method for
NSE equations with EMAC formulation. We use the following notations for the discrete norms. For $\bv^n \in H^p(\Omega)$, we define :
\begin{equation*}
\norm{|\bv|}_{\infty,p}:=\max_{0\leq n\leq N}\|\bv^n\|_p, \quad \norm{|\bv|}_{m,p}:=\bigg(\Delta t \sum_{n=0}^{N} \|\bv^n\|_p^m  \bigg)^{\frac{1}{m}}.
\end{equation*}
For the optimal asymptotic error estimation, assume that the following regularity assumptions hold for the exact solution of (\ref{nse}):
\begin{eqnarray}
\begin{array}{rcll}
\bu &\in& L^{\infty}(0,T;H^1(\Omega))\cap H^1(0,T;H^{k+1}(\Omega))\cap H^3(0,T;L^2(\Omega))\cap H^2(0,T;H^1(\Omega)),\\
p &\in& L^2(0,T;H^{s+1}(\Omega)),\label{ras}\\
f &\in& L^2(0,T;L^2(\Omega)).
\end{array}
\end{eqnarray}
\begin{theorem} \label{thm}
		Let $(\bu,p)$ be the solution of (\ref{nse}) such that the regularity assumptions (\ref{ras}) are satisfied. Then, under the following time step condition
		$$\Delta t\leq C(\norm{|\nabla F[\bu^{n+1}]|}_{\infty,0})^{-1},$$
		the following bound holds for the error $e^n_{\bu}=\bu^n-\bu_h^n$
			 \begin{eqnarray}
		\lefteqn{\norm {e_{\bu}^{N}}^{2} + \dfrac{1}{3}\sum_{n=1}^{N-1} \norm {e_{\bu}^{n+1}-2\mathbf e_{\bu}^{n}+\mathbf e_{\bu}^{n-1}}_{F}^{2}+\dfrac{2\Delta t\nu }{3}\sum_{n=1}^{N-1}\norm {\nabla  F[e_{\bu}^{n+1}]}^{2}} \nonumber
		\\
		&\leq&  \bigg(\dfrac{1}{3}\bigg)^N
		\norm{e_{\bu}^{0}}^{2} +2N(\norm{\mathbf e_{\bu}^{1}}^{2}+\norm{\mathbf e_{\bu}^{0}}^{2})+	C \bigg[\nu^{-1}h^{2k+2}\norm{|\bu_t|}_{2,k+1}^2 +\nu h^{2k}\norm{|\bu|}_{2,k+1}^2+\nu^{-1}h^{2k}\norm{|P^{n+1}|}_{2,k}^2\nonumber\\&&
		+ \nu^{-1}\Delta t^{{4}}\norm{\bu_{ttt}}^2_{L^2(0,T;L^2(\Omega))}+ \Delta t^{{4}}\big(\nu+\nu^{-1}(\norm{|\nabla F[\bu^{n+1}]|}_{\infty,0}+ \norm{|\nabla \bu^{n+1}|}_{\infty,0}))\norm{\nabla\bu_{tt}}^2_{L^2(0,T;L^2(\Omega))}
		\nonumber\\&&+\nu^{-1}\bigg(\norm{|\nabla F[\bu^{n+1}]|}_{\infty,0}h^{2k+1}\norm{|\bu|}_{2,k+1}^2+\norm{|\nabla F[\bu^{n+1}]|}_{\infty,0}\norm{| F[\bu^{n+1}]|}_{\infty,0}h^{2k}\norm{|\bu|}_{2,k+1}^2\bigg) \bigg]
		\label{err}
		\end{eqnarray}
		where C is a generic constant independent of $h$ and $\Delta t$.
\end{theorem}
Theorem \ref{thm} with the most common choice of inf-sup
stable finite element spaces, like Taylor-Hood element, for the velocity and pressure naturally leads to the following Corollary, proving second order accuracy both in time and space.
\begin{corollary} \label{cor} Under the assumptions of Theorem \ref{thm}, let  $(\bfX_h,Q_h)=(P_2,P_{1})$ be	the Taylor-Hood finite element spaces for velocity and pressure. Then, the asymptotic error estimation satisfies, for all $\Delta t>0$
	\begin{eqnarray}
	\norm {e_{\bu}^{N}}^{2} + \dfrac{1}{3}\sum_{n=1}^{N-1} \norm {e_{\bu}^{n+1}-2\mathbf e_{\bu}^{n}+\mathbf e_{\bu}^{n-1}}_{F}^{2}+\dfrac{2\Delta t\nu }{3}\sum_{n=1}^{N-1}\norm {\nabla  F[e_{\bu}^{n+1}]}^{2} \nonumber
\\
	\leq C \bigg(h^4+\Delta t^4+\|\be^0\|^2+\|\be^1\|^2\bigg)\nonumber.
	\end{eqnarray}
\end{corollary}
\begin{proof}
	Application of the approximation properties (\ref{i1})-(\ref{i2}) on the right hand side of (\ref{err}) and the regularity assumption (\ref{ras}) gives the required result.
\end{proof}
We now give the proof of our main theorem.
\begin{proof}
 The proof starts by deriving the error equations. Denote $\bu^{n+1}=\bu(t^{n+1})$. At time $t^{n+1}$, the true solution of the NSE (\ref{nse}) satisfies
 \begin{eqnarray}
 \bigg(\dfrac{3\bu^{n+1}-4\bu^n+\bu^{n-1}}{2\Delta t},\bv_h\bigg)+\nu\big(\nabla  F[ \bu^{n+1}],\nabla \bv_h\big)+c\big( F[ \bu^{n+1}], F[ \bu^{n+1}],\bv_h\big)-\big(P^{n+1},\nabla\cdot\bv_h)\\
 =\big(\bff^{n+1},\bv_h\big)+ Intp(\bu,\bv_h)\label{eh}
 \end{eqnarray}
 for all $\bv_h \in \bfV_h$ where
 	\begin{eqnarray}
 Intp(\bu^{n+1},\bv_h)=\bigg(\dfrac{3\bu^{n+1}-4\bu^n+\bu^{n-1}}{2\Delta t}-\bu_t^{n+1},\bv_h\bigg)+\nu(\nabla F[\bu^{n+1}]-\bu^{n+1}, \bv_h )\nonumber\\+c\big(F[\bu^{n+1}],F[\bu^{n+1}],\bv_h\big)-c\big(\bu^{n+1},\bu^{n+1},\bv_h\big)
 \end{eqnarray}
 denotes the local truncation error.\\
  Subtracting (\ref{E}) from (\ref{eh}) yields
 \begin{eqnarray}
 \bigg(\dfrac{3e_{\bu}^{n+1}-4e_{\bu}^n+e_{\bu}^{n-1}}{2\Delta t},\bv_h\bigg)+\nu\big(\nabla  F[e_{\bu}^{n+1}],\nabla \bv_h\big) +c\big( F[ \bu^{n+1}], F[ \bu^{n+1}],\bv_h\big)\nonumber\\-c\big( F[ \bu_h^{n+1}], F[ \bu_h^{n+1}],\bv_h\big)-\big(P^{n+1},\nabla\cdot\bv_h)
 =Intp(\bu,\bv_h)\label{eh1}
 \end{eqnarray}
 	Decompose the error as
 \begin{eqnarray}
 e_{\bu}^n=\bu(t^n)-\bu_h^n=\big(\bu(t^n)-I^h\bu^n\big)+\big(I^h\bu^n-\bu_h^n\big)= \eta^n+\phi_h^n.\label{se}
 \end{eqnarray}
 where $I^h\bu^n$  is an interpolant of $\bu^n$
 in $\bfV_{h}$.
  Choosing $\bv_h=F[ \phi_h^{n+1}]$ in (\ref{eh1}), using the error decomposition and Lemma \ref{lem:inn}, it follows that
 	\begin{eqnarray}
 \frac{1}{\Delta t}\norm{ \begin{bmatrix}
 	\phi^{n+1}_h \\
 	\phi^{n}_h
 	\end{bmatrix}} _{G}^{2} - \frac{1}{\Delta t}\norm{ \begin{bmatrix}
 	\phi^{n}_h \\
 	\phi^{n-1}_h
 	\end{bmatrix}}_{G}^{2}+
 \frac{1}{4\Delta t}\norm{\phi_h^{n+1}-2\phi_h^n+\phi_h^{n-1}}^2_F+\nu\norm {\nabla F[ \phi_h^{n+1}]}^2\nonumber\\ = \bigg(\dfrac{3 \eta^{n+1}-4 \eta^n+ \eta^{n-1}}{2\Delta t}, F[ \phi_h^{n+1}]\bigg)+\nu\big(\nabla  F[\eta^{n+1}],\nabla F[ \phi_h^{n+1}]\big) \nonumber\\ -c\big( F[ \bu^{n+1}], F[ \bu^{n+1}], F[ \phi_h^{n+1}]\big)+c\big( F[ \bu_h^{n+1}], F[ \bu_h^{n+1}], F[ \phi_h^{n+1}]\big)\nonumber\\-\big(P^{n+1},\nabla\cdot F[ \phi_h^{n+1}]\big)+Intp(\bu, F[ \phi_h^{n+1}])\label{rhs}
 \end{eqnarray}
Next, estimate the terms on the right hand side of  (\ref{rhs}). The first two terms are bounded by applying  Cauchy-Schwarz and Young's inequality:
	\begin{eqnarray}
\bigg|-\bigg(\dfrac{3\eta^{n+1}-4\eta^n+\eta^{n-1}}{2\Delta t}, F[\phi_h^{n+1}]\bigg)\bigg|
&\leq&\norm{\dfrac{3\eta^{n+1}-4\eta^n+\eta^{n-1}}{2\Delta t}}\norm{  F[\phi_h^{n+1}]}^2\nonumber\\
&\leq& \dfrac{C\nu^{-1}}{\Delta t}\int_{t^{n-1}}^{t^{n+1}}\|\eta_t\|^2dt+\dfrac{\nu}{28}\| \nabla F[\phi_h^{n+1}]\|^2 \label{t1}
\end{eqnarray}
\begin{eqnarray}
\big|\nu(\nabla F[\eta^{n+1}],\nabla F[\phi_h^{n+1}])\big|
&\leq&\nu\|\nabla F[\eta^{n+1}]\|\|\nabla F[\phi_h^{n+1}]\|\nonumber\\
&\leq& C\nu\|\nabla F[\eta^{n+1}]\|^2+\dfrac{\nu}{28}\|\nabla F[\phi_h^{n+1}]\|^2.\label{t5}
\end{eqnarray}
Following Theorem 3.2 in [43], the nonlinear terms are estimated as,
\begin{eqnarray}
\lefteqn{\bigg| -c\big( F[ \bu^{n+1}], F[ \bu^{n+1}], F[ \phi_h^{n+1}]\big)+c\big( F[ \bu_h^{n+1}], F[ \bu_h^{n+1}], F[ \phi_h^{n+1}]\big)\bigg|}\nonumber\\&\leq& C\nu^{-1}\bigg(\norm{\nabla F[\bu^{n+1}]}^2\norm{ F[\eta^{n+1}]}\norm{\nabla F[\eta^{n+1}]}+\norm{\nabla F[\bu^{n+1}]}\norm{ F[\bu^{n+1}]}\norm{\nabla F[\eta^{n+1}]}^2\bigg)\nonumber\\&&+C\nu^{-1}\norm{\nabla F[\bu^{n+1}]}^2\norm{F[ \phi_h^{n+1}]}^2+\dfrac{\nu}{28}\|\nabla F[\phi_h^{n+1}]\|^2.
\end{eqnarray}
For the pressure term, use the fact that $(\nabla\cdot\phi_h,q_h)=0,\, \forall \phi_h \in \bfV_h$ together with Cauchy-Schwarz and Young's inequalities to get
\begin{eqnarray}
\big|P^{n+1},\nabla\cdot \phi_h^{n+1})\big|&=& \big|(P^{n+1}-q_h,\nabla\cdot \phi_h^{n+1})\big|\nonumber\\
&\leq& C\nu^{-1}\big\|\inf_{q_h\in Q_h}\big\|P^{n+1}-q_h\big\|^2+\dfrac{\nu}{28}\|\nabla F[ \phi_h^{n+1}]\|^2.\label{t8}
\end{eqnarray}
We proceed to bound the terms in the local truncation error $Intp(\bu, F[ \phi_h^{n+1}])$. For the first two terms of $Intp(\bu, F[ \phi_h^{n+1}])$,
apply the Cauchy-Schwarz and Young's inequalities together with
the integral remainder form of Taylor's theorem
 to obtain
	\begin{eqnarray}
\lefteqn{\bigg|\bigg(\dfrac{3\bu^{n+1}-4\bu^n+\bu^{n-1}}{2\Delta t}-\bu_t^{n+1},F[ \phi_h^{n+1}]\bigg)\bigg|}\nonumber\\
&\leq& \norm{\dfrac{3\bu^{n+1}-4\bu^n+\bu^{n-1}}{2\Delta t}-\bu_t(t^{n+1})}\norm{F[ \phi_h^{n+1}]}\nonumber\\
&\leq& C\Delta t^{{3}}\nu^{-1}\int_{t_{n-1}}^{t_{n+1}}\|\bu_{ttt}\|^2dt+\dfrac{\nu}{28}\|\nabla F[ \phi_h^{n+1}]\| \label{t10}
\end{eqnarray}
\begin{eqnarray}
\nu(\nabla(F[ \bu^{n+1}]-\bu_{n+1}),\nabla F[ \phi_h^{n+1}])
\leq C\nu\norm{\nabla(F[ \bu^{n+1}]-\bu_{n+1})}^2+\dfrac{\nu}{28}\norm{\nabla F[ \phi_h^{n+1}]}^{2}\nonumber\\
\leq C\nu\Delta t^{3}\int_{t_{n-1}}^{t_{n+1}}\norm{\nabla\bu_{tt}}^{2}dt+\dfrac{\nu}{28}\norm{\nabla F[ \phi_h^{n+1}]}^{2}.\nonumber
\end{eqnarray}
To bound the convective terms in $Intp(\bu, F[ \phi_h^{n+1}])$, we first rearrange the terms. Adding and subtracting terms for the convective terms and using the definition of the EMAC formulation gives
\begin{eqnarray}
\lefteqn{c\big(F[\bu^{n+1}],F[\bu^{n+1}],F[ \phi_h^{n+1}]\big)-c\big(\bu^{n+1},\bu^{n+1},F[ \phi_h^{n+1}]\big)}\nonumber\\
&=&c\big(F[\bu^{n+1}]-\bu^{n+1},F[\bu^{n+1}],F[ \phi_h^{n+1}]\big)+c\big(\bu^{n+1},F[\bu^{n+1}]-\bu^{n+1},F[ \phi_h^{n+1}]\big)\nonumber\\
&=&(F[\bu^{n+1}]\cdot \nabla (F[\bu^{n+1}]-\bu^{n+1}),F[ \phi_h^{n+1}])+ (F[ \phi_h^{n+1}] \cdot \nabla F[\bu^{n+1}]-\bu^{n+1}),F[\bu^{n+1}])\nonumber\\ &-&  (  (F[\bu^{n+1}]-\bu^{n+1})  \cdot \nabla F[\bu^{n+1}], F[ \phi_h^{n+1}])- (  (F[\bu^{n+1}]-\bu^{n+1})  \cdot \nabla F[ \phi_h^{n+1}], F[\bu^{n+1}])\nonumber\\
&+& (  (F[\bu^{n+1}]-\bu^{n+1})  \cdot \nabla \bu^{n+1}, F[ \phi_h^{n+1}])+ (F[ \phi_h^{n+1}] \cdot \nabla \bu^{n+1}, F[\bu^{n+1}]-\bu^{n+1}))\nonumber\\
&-& (\bu^{n+1}\cdot \nabla (F[\bu^{n+1}]-\bu^{n+1}),F[ \phi_h^{n+1}])- ( \bu^{n+1}\cdot \nabla F[ \phi_h^{n+1}] ,(F[\bu^{n+1}]-\bu^{n+1})).\label{eh2}
\end{eqnarray}
Then, the convective terms in (\ref{eh2}) are estimated by applying the Cauchy-Schwarz and Young's inequalities together with
the integral remainder form of Taylor's theorem as
\begin{eqnarray}
\lefteqn{|(F[\bu^{n+1}]\cdot \nabla (F[\bu^{n+1}]-\bu^{n+1}),F[ \phi_h^{n+1}])|}\nonumber\\&\leq& C\nu^{-1}\norm{\nabla F[\bu^{n+1}]}^2\norm{\nabla (F[\bu^{n+1}]-\bu^{n+1})}^2+\dfrac{\nu}{28}\norm{\nabla F[ \phi_h^{n+1}]}^{2}\nonumber\\
&\leq& C\nu^{-1}\Delta t^{3}\norm{\nabla F[\bu^{n+1}]}^2\int_{t_{n-1}}^{t_{n+1}}\norm{\nabla\bu_{tt}}^{2}dt+\dfrac{\nu}{28}\norm{\nabla F[ \phi_h^{n+1}]}^{2},\label{eh3}
\end{eqnarray}
\begin{eqnarray}
\lefteqn{|(F[ \phi_h^{n+1}] \cdot \nabla F[\bu^{n+1}]-\bu^{n+1}),F[\bu^{n+1}])|}\nonumber\\&\leq& C\nu^{-1}\norm{\nabla F[\bu^{n+1}]}^2\norm{\nabla (F[\bu^{n+1}]-\bu^{n+1})}^2+\dfrac{\nu}{28}\norm{\nabla F[ \phi_h^{n+1}]}^{2}\nonumber\\
&\leq& C\nu^{-1}\Delta t^{3}\norm{\nabla F[\bu^{n+1}]}^2\int_{t_{n-1}}^{t_{n+1}}\norm{\nabla\bu_{tt}}^{2}dt+\dfrac{\nu}{28}\norm{\nabla F[ \phi_h^{n+1}]}^{2},
\end{eqnarray}
\begin{eqnarray}
\lefteqn{|- (  (F[\bu^{n+1}]-\bu^{n+1})  \cdot \nabla F[\bu^{n+1}], F[ \phi_h^{n+1}])| }\nonumber\\&\leq& C\nu^{-1}\norm{\nabla F[\bu^{n+1}]}^2\norm{\nabla (F[\bu^{n+1}]-\bu^{n+1})}^2+\dfrac{\nu}{28}\norm{\nabla F[ \phi_h^{n+1}]}^{2}\nonumber\\
&\leq& C\nu^{-1}\Delta t^{3}\norm{\nabla F[\bu^{n+1}]}^2\int_{t_{n-1}}^{t_{n+1}}\norm{\nabla\bu_{tt}}^{2}dt+\dfrac{\nu}{28}\norm{\nabla F[ \phi_h^{n+1}]}^{2},
\end{eqnarray}

\begin{eqnarray}
\lefteqn{|-(  (F[\bu^{n+1}]-\bu^{n+1})  \cdot \nabla F[ \phi_h^{n+1}], F[\bu^{n+1}])|}\nonumber\\&\leq& C\nu^{-1}\norm{\nabla F[\bu^{n+1}]}^2\norm{\nabla (F[\bu^{n+1}]-\bu^{n+1})}^2+\dfrac{\nu}{28}\norm{\nabla F[ \phi_h^{n+1}]}^{2}\nonumber\\
&\leq&C\nu^{-1}\Delta t^{3}\norm{\nabla F[\bu^{n+1}]}^2\int_{t_{n-1}}^{t_{n+1}}\norm{\nabla\bu_{tt}}^{2}dt+\dfrac{\nu}{28}\norm{\nabla F[ \phi_h^{n+1}]}^{2},
\end{eqnarray}
\begin{eqnarray}
\lefteqn{|(  (F[\bu^{n+1}]-\bu^{n+1})  \cdot \nabla \bu^{n+1}, F[ \phi_h^{n+1}])|}\nonumber\\&\leq& C\nu^{-1}\norm{\nabla \bu^{n+1}}^2\norm{\nabla (F[\bu^{n+1}]-\bu^{n+1})}^2+\dfrac{\nu}{28}\norm{\nabla F[ \phi_h^{n+1}]}^{2}\nonumber\\
&\leq& C\nu^{-1}\Delta t^{3}\norm{\nabla \bu^{n+1}}^2\int_{t_{n-1}}^{t_{n+1}}\norm{\nabla\bu_{tt}}^{2}dt+\dfrac{\nu}{28}\norm{\nabla F[ \phi_h^{n+1}]}^{2},
\end{eqnarray}
\begin{eqnarray}
\lefteqn{|(F[ \phi_h^{n+1}] \cdot \nabla \bu^{n+1}, F[\bu^{n+1}]-\bu^{n+1}))|}\nonumber\\&\leq& C\nu^{-1}\norm{\nabla \bu^{n+1}}^2\norm{\nabla (F[\bu^{n+1}]-\bu^{n+1})}^2+\dfrac{\nu}{28}\norm{\nabla F[ \phi_h^{n+1}]}^{2}\nonumber\\
&\leq& C\nu^{-1}\Delta t^{3}\norm{\nabla \bu^{n+1}}^2\int_{t_{n-1}}^{t_{n+1}}\norm{\nabla\bu_{tt}}^{2}dt+\dfrac{\nu}{28}\norm{\nabla F[ \phi_h^{n+1}]}^{2},
\end{eqnarray}
\begin{eqnarray}
\lefteqn{|-(\bu^{n+1}\cdot \nabla (F[\bu^{n+1}]-\bu^{n+1}),F[ \phi_h^{n+1}]))|}\nonumber\\&\leq& C\nu^{-1}\norm{\nabla \bu^{n+1}}^2\norm{\nabla (F[\bu^{n+1}]-\bu^{n+1})}^2+\dfrac{\nu}{28}\norm{\nabla F[ \phi_h^{n+1}]}^{2}\nonumber\\
&\leq& C\nu^{-1}\Delta t^{3}\norm{\nabla \bu^{n+1}}^2\int_{t_{n-1}}^{t_{n+1}}\norm{\nabla\bu_{tt}}^{2}dt+\dfrac{\nu}{28}\norm{\nabla F[ \phi_h^{n+1}]}^{2},
\end{eqnarray}
\begin{eqnarray}
\lefteqn{|-(\bu^{n+1}\cdot \nabla F[ \phi_h^{n+1}], (F[\bu^{n+1}]-\bu^{n+1}))|}\nonumber\\&\leq& C\nu^{-1}\norm{\nabla \bu^{n+1}}^2\norm{\nabla (F[\bu^{n+1}]-\bu^{n+1})}^2+\dfrac{\nu}{28}\norm{\nabla F[ \phi_h^{n+1}]}^{2}\nonumber\\
&\leq& C\nu^{-1}\Delta t^{3}\norm{\nabla \bu^{n+1}}^2\int_{t_{n-1}}^{t_{n+1}}\norm{\nabla\bu_{tt}}^{2}dt+\dfrac{\nu}{28}\norm{\nabla F[ \phi_h^{n+1}]}^{2}.\label{eh4}
\end{eqnarray}
Inserting (\ref{eh3})-(\ref{eh4}) into (\ref{eh}) gives
	\begin{eqnarray}
\lefteqn{\frac{1}{\Delta t}\norm{ \begin{bmatrix}
	\phi^{n+1}_h \\
	\phi^{n}_h
	\end{bmatrix}} _{G}^{2} -  \frac{1}{\Delta t}\norm{ \begin{bmatrix}
	\phi^{n}_h \\
	\phi^{n-1}_h
	\end{bmatrix}} _{G}^{2}+
\frac{1}{4\Delta t}\norm{\phi_h^{n+1}-2\phi_h^n+\phi_h^{n1}}^2_F+\frac{\nu}{2}\norm {\nabla F[ \phi_h^{n+1}]}^2}
\nonumber\\
&\leq& \frac{ C\nu^{-1}}{\Delta t}\int_{t^{n-1}}^{t^{n+1}}\|\eta_t\|^2dt + C\nu\|\nabla F[\eta^{n+1}]\|^2+C\nu^{-1}\big\|\inf_{q_h\in Q_h}\big\|P^{n+1}-q_h\big\|^2\nonumber\\
&&+ C\nu^{-1}\Delta t^{{3}}\int_{t_{n-1}}^{t_{n+1}}\|\bu_{ttt}\|^2dt+ C\nu\Delta t^{3}\int_{t_{n-1}}^{t_{n+1}}\norm{\nabla\bu_{tt}}^{2}dt\nonumber\\&&+C\nu^{-1}\Delta t^{3}\big(\norm{\nabla F[\bu^{n+1}]}^2+\norm{\nabla \bu^{n+1}}^2\big)\int_{t_{n-1}}^{t_{n+1}}\norm{\nabla\bu_{tt}}^{2}dt\nonumber\\&&+C\nu^{-1}\bigg(\norm{\nabla F[\bu^{n+1}]}^2\norm{ F[\eta^{n+1}]}\norm{\nabla F[\eta^{n+1}]}+\norm{\nabla F[\bu^{n+1}]}\norm{ F[\bu^{n+1}]}\norm{\nabla F[\eta^{n+1}]}^2\bigg)\nonumber\\&&+C\nu^{-1}\norm{\nabla F[\bu^{n+1}]}^2\norm{F[ \phi_h^{n+1}]}^2\nonumber
\end{eqnarray}
Multiplying by $\Delta t$, summing from $n=1$ to $n=N-1$ and using approximation properties (\ref{ap1})-(\ref{ap10}) yields
	\begin{eqnarray}
\lefteqn{\norm{ \begin{bmatrix}
	\phi^{N}_h \\
	\phi^{N-1}_h
	\end{bmatrix}} _{G}^{2} +
\frac{1}{4}\norm{\phi_h^{n+1}-2\phi_h^n+\phi_h^{n-1}}^2_F+\frac{\nu\Delta t}{2}\norm {\nabla F[ \phi_h^{n+1}]}^2}
\nonumber\\
&\leq& \norm{ \begin{bmatrix}
	\phi^{1}_h \\
	\phi^{0}_h
	\end{bmatrix}} _{G}^{2}+ C\bigg(\nu^{-1}h^{2k+2}\norm{|\bu_t|}_{2,k+1}^2 +\nu h^{2k}\norm{|\bu|}_{2,k+1}^2+\nu^{-1}h^{2k}\norm{|P^{n+1}|}_{2,k}^2\bigg)\nonumber\\&&
+ C\nu^{-1}\Delta t^{{4}}\norm{\bu_{ttt}}^2_{L^2(0,T;L^2(\Omega))}+ C\Delta t^{{4}}\big(\nu+\nu^{-1}(\norm{|\nabla F[\bu^{n+1}]|}_{\infty,0}+ \norm{|\nabla \bu^{n+1}|}_{\infty,0}))\norm{\nabla\bu_{tt}}^2_{L^2(0,T;L^2(\Omega))}
\nonumber\\&&+C\nu^{-1}\bigg(\norm{|\nabla F[\bu^{n+1}]|}_{\infty,0}h^{2k+1}\norm{|\bu|}_{2,k+1}^2+\norm{|\nabla F[\bu^{n+1}]|}_{\infty,0}\norm{| F[\bu^{n+1}]|}_{\infty,0}h^{2k}\norm{|\bu|}_{2,k+1}^2\bigg)\nonumber\\&&+C\nu^{-1}\norm{|\nabla F[\bu^{n+1}]|}_{\infty,0}\Delta t \sum_{n=1}^{N-1}\norm{F[ \phi_h^{n+1}]}^2
\end{eqnarray}
	Applying the discrete Gronwall inequality with the assumption
	$$\Delta t\leq C(\norm{|\nabla F[\bu^{n+1}]|}_{\infty,0})^{-1}$$
	and using Lemma \ref{lem:gnorm} produces
	 \begin{eqnarray}
\lefteqn{\norm {\bphi^{N}_{h}}^{2} + \dfrac{1}{3}\sum_{n=1}^{N-1} \norm {\bphi^{n+1}_{h}-2\mathbf \bphi^{n}_{h}+\mathbf \bphi^{n-1}_{h}}_{F}^{2}+\dfrac{2\Delta t\nu }{3}\sum_{n=1}^{N-1}\norm {\nabla  F[\bphi_{h}^{n+1}]}^{2}} \nonumber
\\
&\leq& 	 \bigg(\dfrac{1}{3}\bigg)^N
\norm{
	\bphi^{0}_{h}}^{2} +2N(\norm{\mathbf \bphi^{1}_h}^{2}+\norm{\mathbf \bphi^{0}_h}^{2})+C \bigg[\nu^{-1}h^{2k+2}\norm{|\bu_t|}_{2,k+1}^2 +\nu h^{2k}\norm{|\bu|}_{2,k+1}^2+\nu^{-1}h^{2k}\norm{|P^{n+1}|}_{2,k}^2\nonumber\\&&
+ \nu^{-1}\Delta t^{{4}}\norm{\bu_{ttt}}^2_{L^2(0,T;L^2(\Omega))}+ \Delta t^{{4}}\big(\nu+\nu^{-1}(\norm{|\nabla F[\bu^{n+1}]|}_{\infty,0}+ \norm{|\nabla \bu^{n+1}|}_{\infty,0}))\norm{\nabla\bu_{tt}}^2_{L^2(0,T;L^2(\Omega))}
\nonumber\\&&+\nu^{-1}\bigg(\norm{|\nabla F[\bu^{n+1}]|}_{\infty,0}h^{2k+1}\norm{|\bu|}_{2,k+1}^2+\norm{|\nabla F[\bu^{n+1}]|}_{\infty,0}\norm{| F[\bu^{n+1}]|}_{\infty,0}h^{2k}\norm{|\bu|}_{2,k+1}^2\bigg)\bigg]
\label{sta}
\end{eqnarray}
The final result follows from the triangle inequality.
\end{proof}
\section{Numerical Experiments}
In this section, we perform three different numerical experiments to test the effectiveness of the proposed Algorithm (\ref{a1})-(\ref{a3}) and compare the results with the non-filtered BE-EMAC scheme (step 1 without step 2). The first test confirms the order of convergence rates predicted in Corollary \ref{cor} for an analytic test problem with a known solution. In second example, we check the energy, momentum and angular momentum conservation of the EMAC-FILTERED scheme in a so-called Gresho problem. In last test, we studied a typical  benchmark problems of flow around a cylinder to demonstrate the superiority of  EMAC-FILTERED method  over BE-EMAC scheme.  All simulations are carried out with the Taylor-Hood finite element pairs $(\bfX_h, Q_h)=(P_2,P_1)$ for velocity and pressure on conforming triangular grids. The computations are performed with the public license finite element software package FreeFem++ \cite{hec}.

\subsection{Convergence Rates}
In this part, we verify the expected convergence rates of our numerical scheme (\ref{a1})-(\ref{a3}) described by Corollary \ref{cor}. For this purpose, we pick the analytical solution:
\begin{eqnarray}
\bu=\begin{pmatrix}
cos(y)e^t\\sin(x)e^t
\end{pmatrix} ,\quad  p=(x-y)(1 + t)\nonumber
\end{eqnarray}
with the kinematic viscosity $\nu=1$ and from which the external force is determined so that (\ref{nse}) is satisfied. Computations are performed in the unit square domain $\Omega=[0,1]^2$.  The boundary conditions are enforced to be the true solution. The approximate solutions are computed on successive mesh refinements and the velocity errors are measured in the discrete norm $L^2(0,T;H^1(\Omega))$
 $$\|\textbf{u}-\textbf{u}^h\|_{2,1}=\left\{\Delta t \sum_{n=1}^{N}\|\nabla\big(\textbf{u}(t^n)-\textbf{u}_{n}^{h}\big)\|^{2}\right\}^{1/2}.$$
 To test the spatial convergence, we fixed the time step as $\Delta t = 0.00001$ with an end time $t= 10^{-4}$ to isolate the spatial error and calculate the errors for varying $h$. Results for errors and rates are shown in Table \ref{table:tab1}. In a similar manner, we fix the mesh size to $h=1/128$ to compute temporal errors and convergence rates by using different time steps with an end time of $t=1$, see Table \ref{table:tab2}. One can observe second order accuracy both in time and space, which is the optimal convergence rate found in Corollary \ref{cor}. Thus we can conclude that the addition of time filtering not only increases the time accuracy, but also does not degrade the spatial order of convergence.

\begin{table}[h!]
	 \begin{center}
	 \begin{minipage}[b]{60mm}

 		\begin{tabular}{|c|c|c|}
 			\hline
 			$h$ & $\|u-u^h\|_{2,1}$ & Rate \\
 			\hline
 			1/4&2.32618e-6 &--   \\
 			\hline
 			1/8&5.80867e-7 &2.00209     \\
 			\hline
 			1/16&1.44272e-7 &2.00938     \\
 			\hline
 			1/32&3.5518e-8 &2.0221   \\
 			\hline
 			1/64&9.56472e-9 &1.89275       \\
 			\hline
 		\end{tabular}
 		\caption{Spatial errors and rates}
 		\label{table:tab1}
 	 \end{minipage} \begin{minipage}[b]{60mm}

		\begin{tabular}{|c|c|c|}
			\hline
			$\Delta t$ & $\|u-u^h\|_{2,1}$ & Rate \\
			\hline
			1/4&0.0281784 &--   \\
			\hline
			1/8&0.00693954 &2.02165     \\
			\hline
			1/16&0.00124549 &2.47811    \\
			\hline
			1/32&0.000227284 &2.45412   \\
			\hline
			1/64&4.08335e-5 &2.4767       \\
			\hline
		\end{tabular}
		\caption{Temporal errors and rates}
		\label{table:tab2}
\end{minipage}
 \end{center}
 \end{table}
\subsection{Gresho Problem}
The second experiment we consider is Gresho problem, which is also referred to as the "standing vortex problem" \cite{P90, RB}. We aim here to numerically verify that the quantities mentioned in Section \ref{seccon} are conserved by EMAC-FILTERED scheme. The simulation starts with an initial condition $\bu_0$ that is an exact solution of the steady Euler equations. On the domain $\Omega=(-0.5,0.5)^2$ with $r=\sqrt{x^2+y^2}$, the velocity and pressure solutions are defined by

\begin{eqnarray}
r\leq 0.2 &:& \bigg\{\bu=\left(
\begin{array}{c}
-5y \\
5x
\end{array}\right),
p=12.5r^2 + C_1 \nonumber\\
0.2\leq r\leq 0.4 &:& \bigg\{\bu=\left(
\begin{array}{c}
\frac{2y}{r}+5y \\
\frac{2x}{r}-5x
\end{array}\right),
p=12.5r^2-20r+4\log(r) + C_2 \\
r>0.4 &:& \bigg\{\bu=\left(
\begin{array}{c}
0 \\
0
\end{array}\right),
p=0\nonumber
\end{eqnarray}
where
$$ C_2=(-12.5)(0.4)^2+20(0.4)^2-4\log(0.4)
, \quad \quad C_1=C_2-20(0.2)+4\log(0.2).$$
\begin{figure}[h!]
	\centering
	\includegraphics[width=75mm]{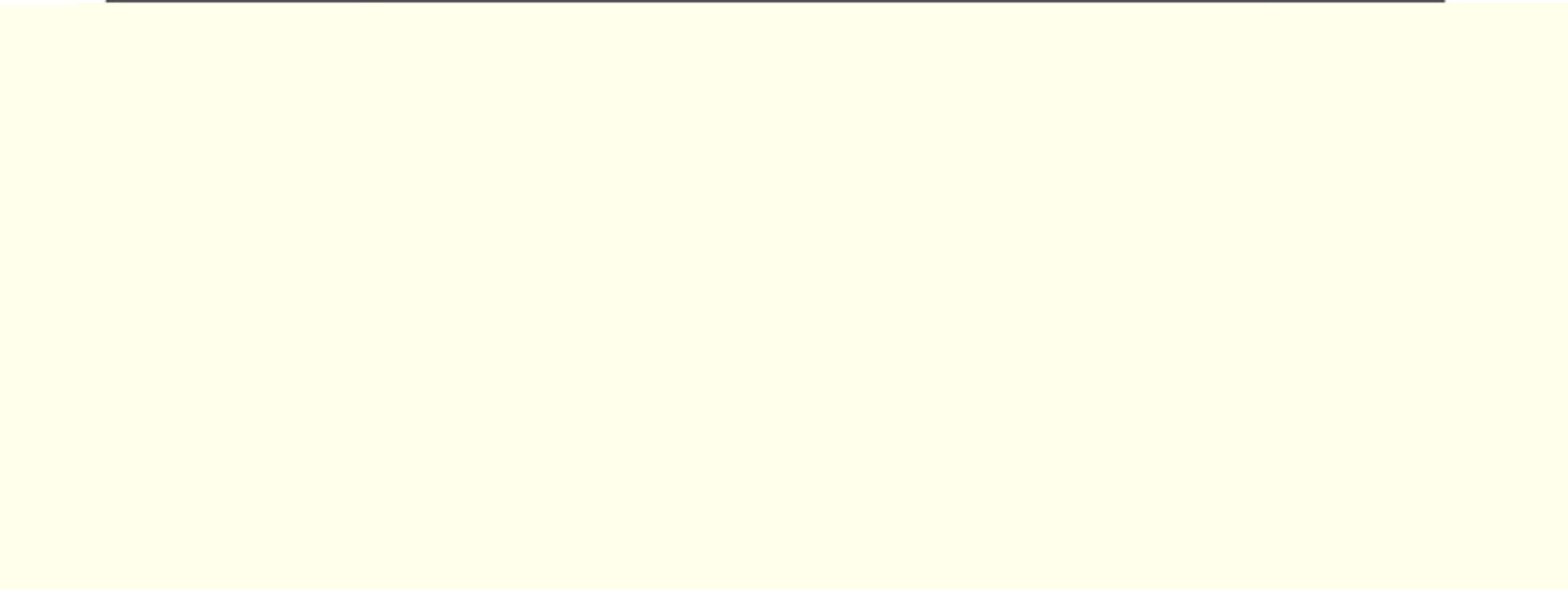}

	\caption{Speed contours of true solution of the Gresho prolem at all times.}
	\label{fig:mesh}
\end{figure}
The speed plot of this initial condition can be seen in Figure \ref{fig:mesh}.\\
We compute solutions of the EMAC-FILTERED and BE-EMAC schemes by using Newton iterations to solve the nonlinear term with $\bff=0,\nu=0$ and no penetration boundary conditions up to $T=8$. The computations are run on a $48\times 48$ uniform mesh with a time step size $\Delta t=0.025$. Since the initial condition is the solution of the steady Euler equation, the accuracy of the method depends on keeping this solution unchanged over time. In addition, since there are no viscosity and external force, the problem is highly suitable to test the conservation properties of a numerical method. Plots of energy, momentum, angular momentum and $L^2$ error versus time of both the EMAC-FILTERED and the BE-EMAC are shown in Figure \ref{fig:Conservation}. We can deduce from the figure that EMAC-FILTERED scheme we consider conserves momentum and angular momentum and accurate as predicted in the theory. Making use of this scheme has no drawbacks in terms of preserving desired physical quantities when compared to BE-EMAC. For longer time intervals, the energy loss of EMAC-FILTERED scheme is slightly better than the BE-EMAC scheme after $t=5$.

In addition, we calculate and compare the
physical dissipation and numerical dissipation to support the conservation properties of the scheme for the same test with $ Re=1000$ over time interval of $[0, 10]$. The results are presented in Figure \ref{fig:dis}. As clearly seen, numerical dissipation, which is almost zero, is much smaller than
the physical dissipation. This shows that the energy loss of the proposed scheme is very low. So we can deduce that, in terms of physical quantities, making use of EMAC-FILTERED scheme has no disadvantage over BE-EMAC scheme and even slightly advantageous in terms of energy loss.

\begin{figure}[h!]
	\centering
	\begin{subfigure}[b]{0.4\linewidth}
		\includegraphics[width=70 mm]{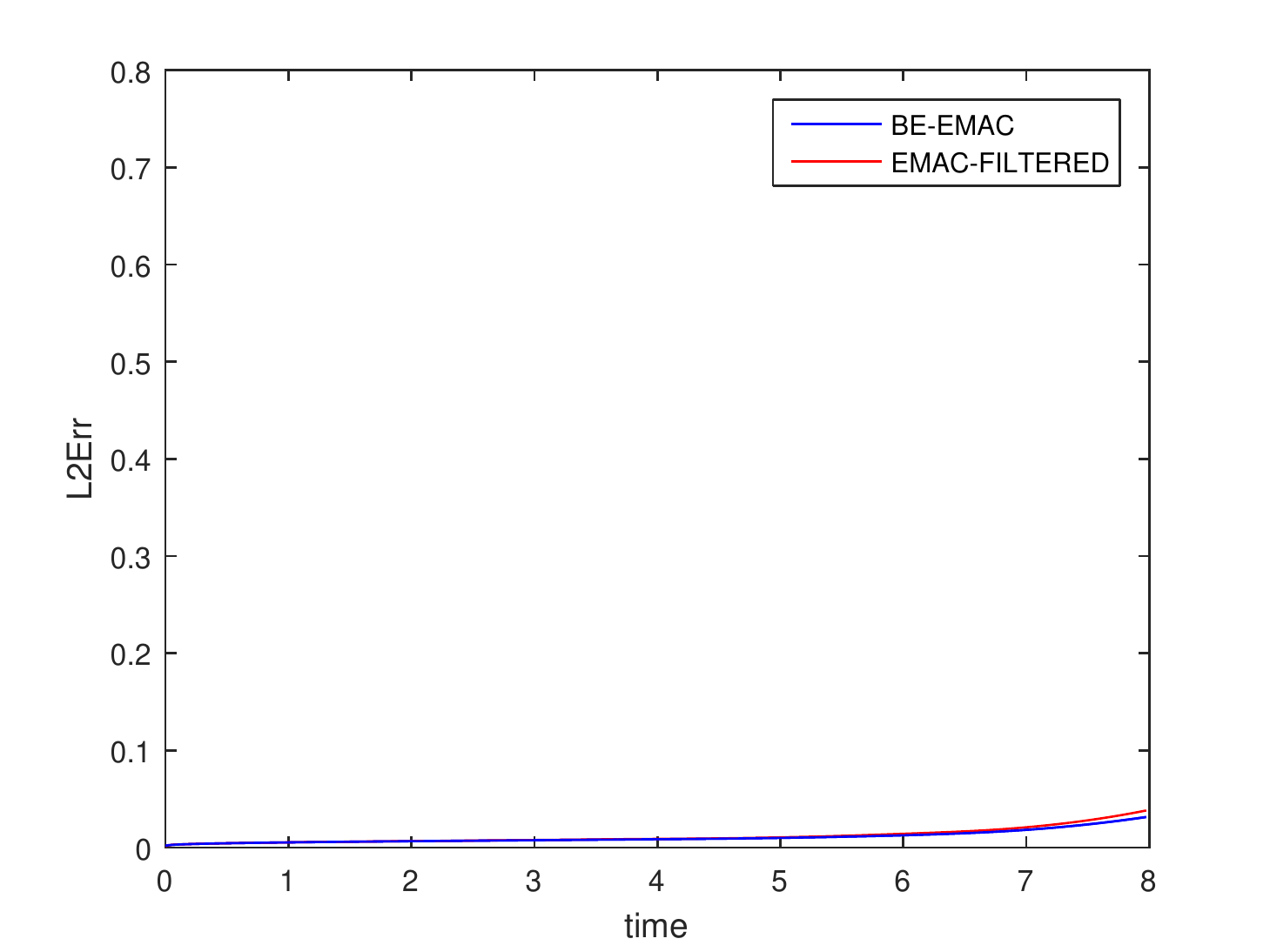}
	\end{subfigure}
	\begin{subfigure}[b]{0.4\linewidth}
		\includegraphics[width=70 mm]{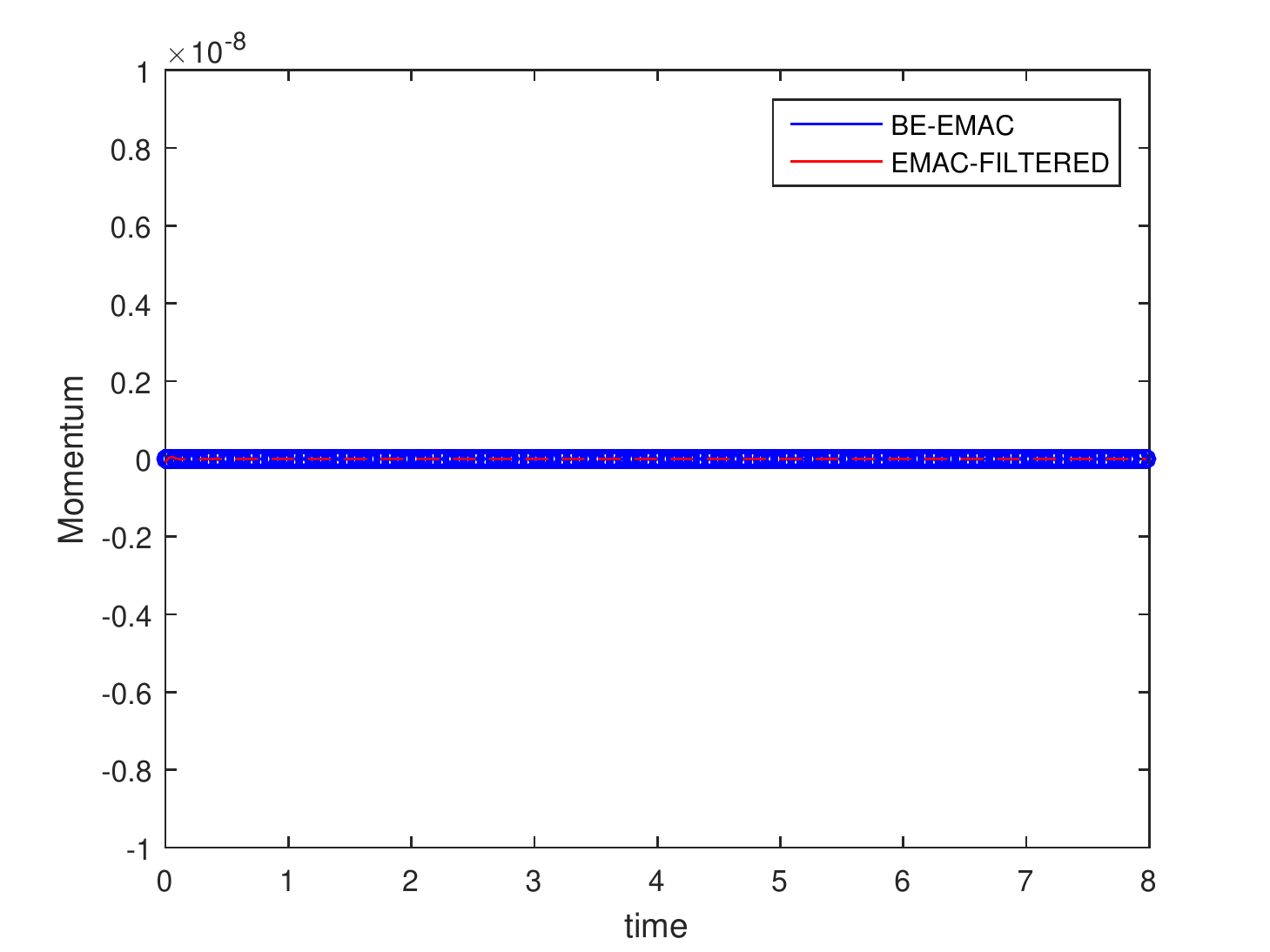}
	\end{subfigure}\\
	\begin{subfigure}[b]{0.4\linewidth}
		\includegraphics[width=70 mm]{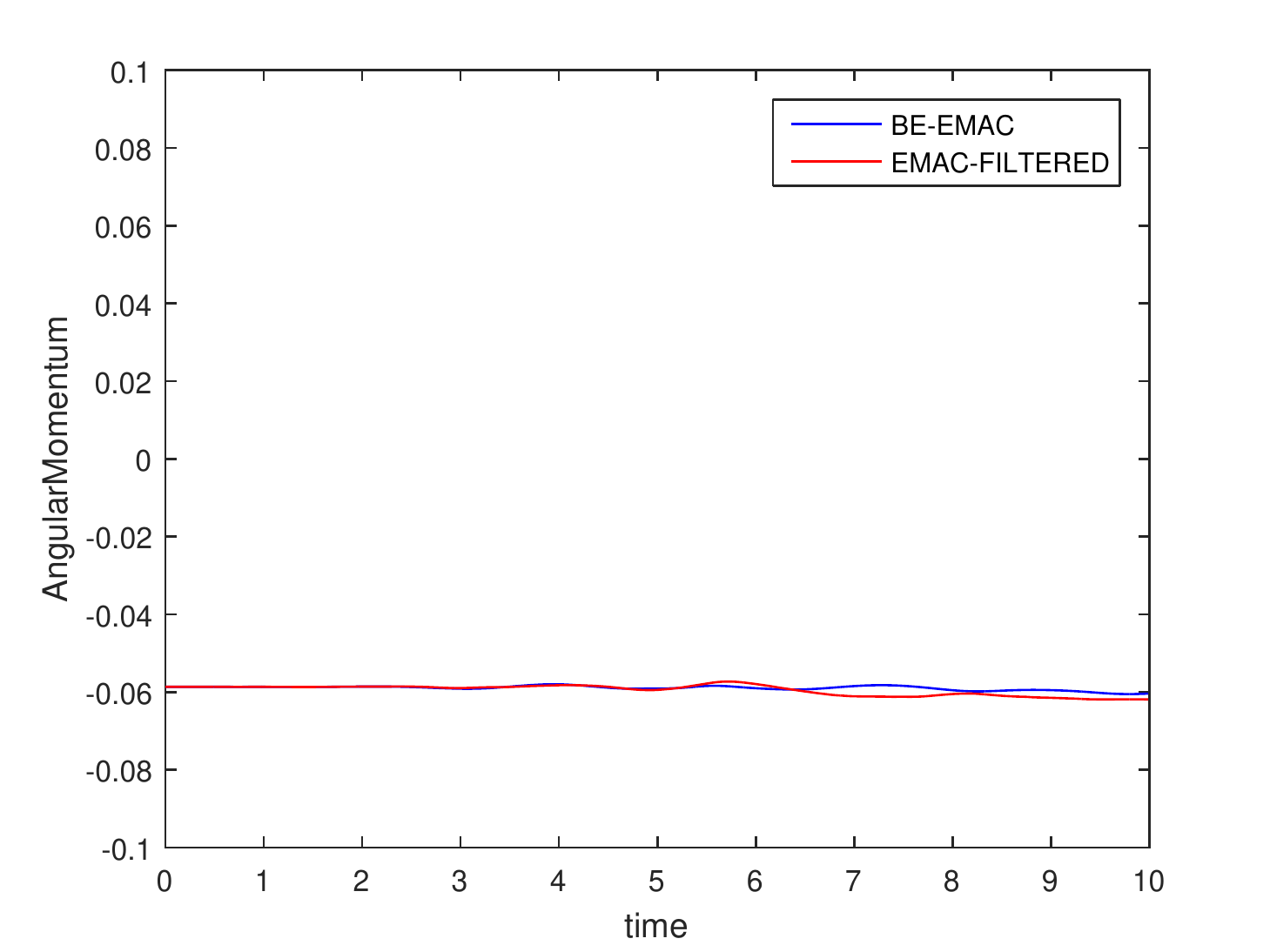}
	\end{subfigure}
\begin{subfigure}[b]{0.4\linewidth}
	\includegraphics[width=70 mm]{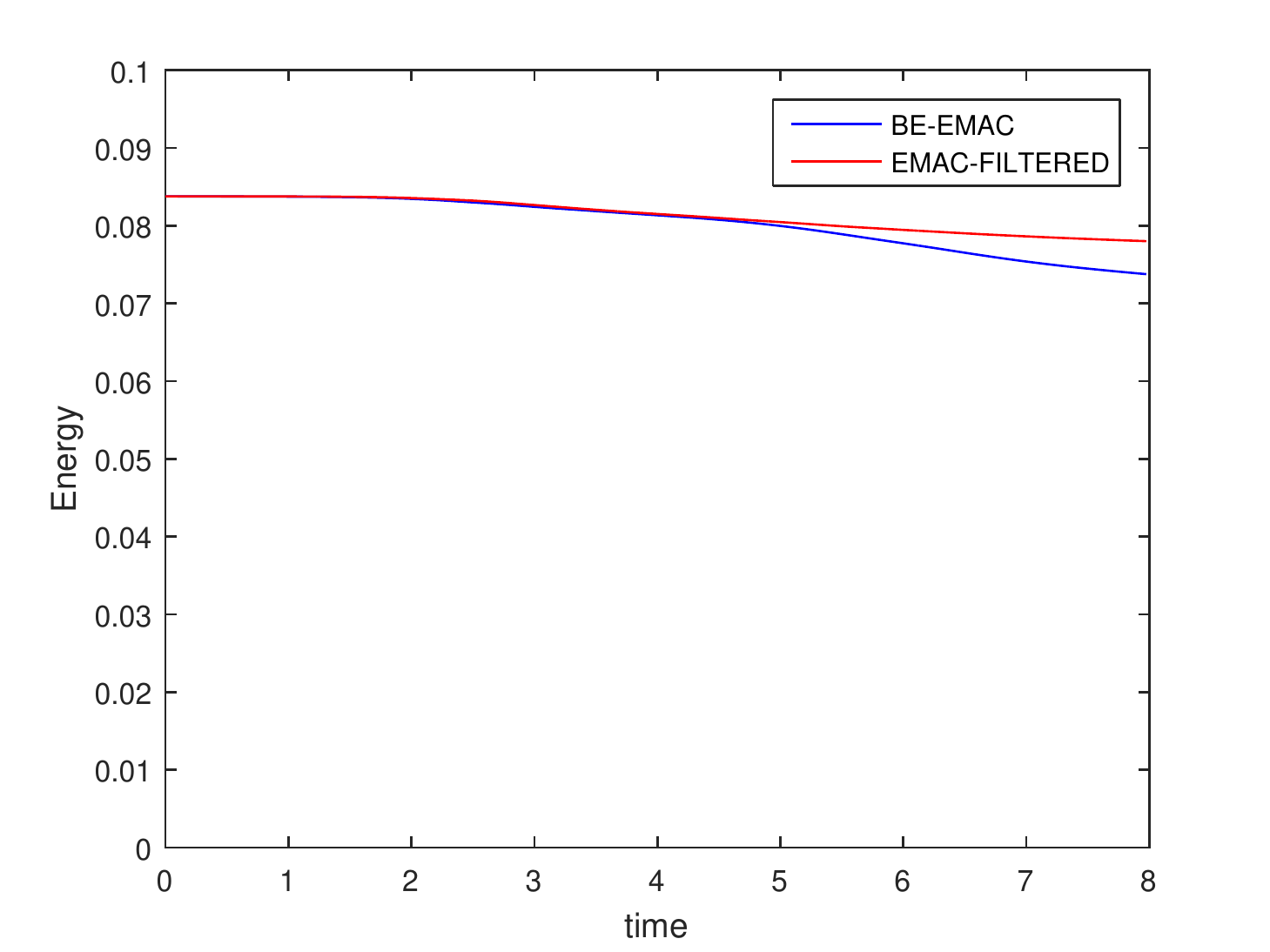}
\end{subfigure}
	\caption{Plots of time versus $L^2$ error, energy, momentum and angular momentum. }
	\label{fig:Conservation}
\end{figure}

\begin{figure}[h!]
	
\centering
	\includegraphics[width=80mm]{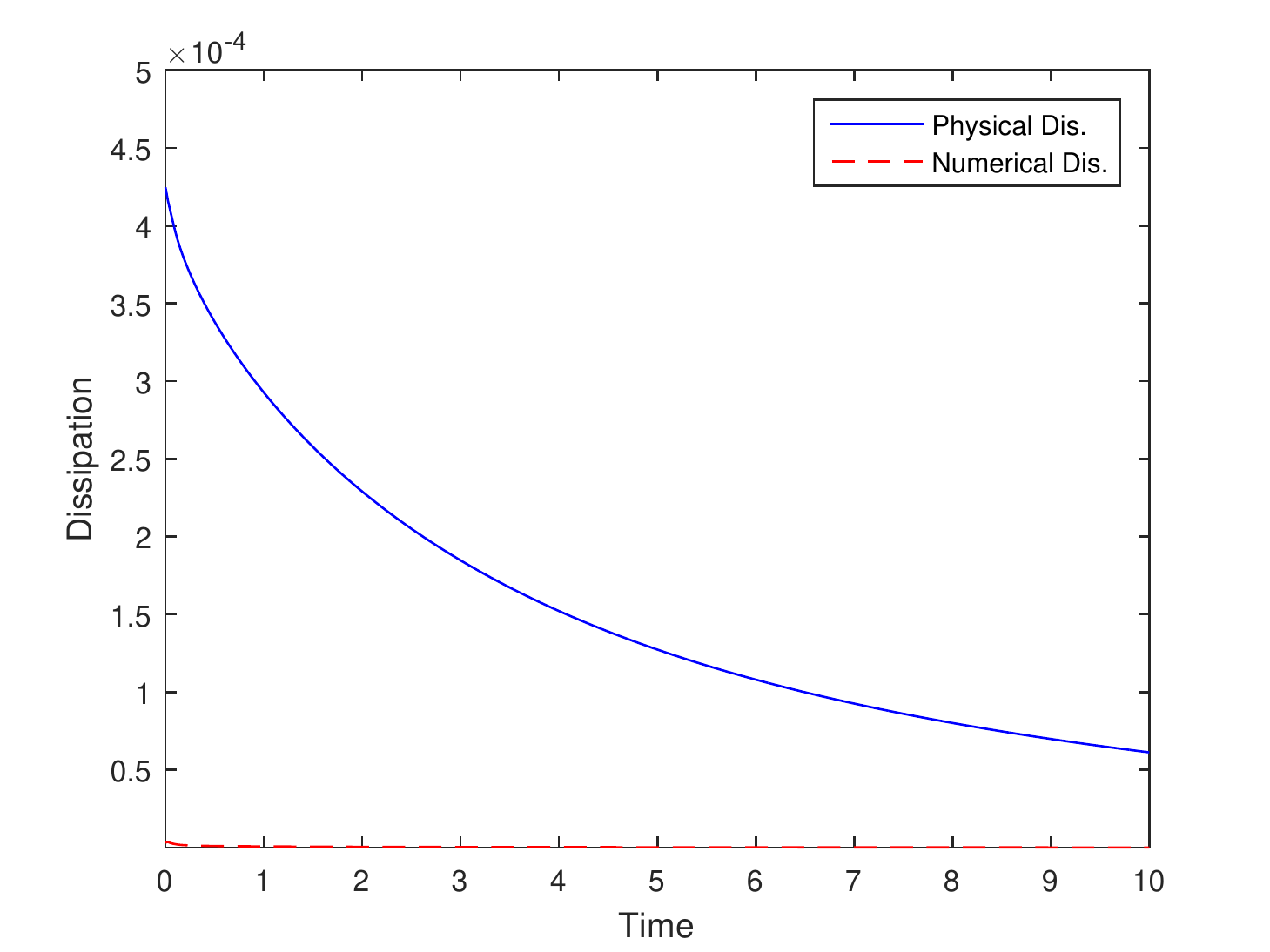}
	\caption{ Numerical and physical dissipation versus time for our scheme}
		\label{fig:dis}
\end{figure}
\subsection{Flow Around a Cylinder}
In the next experiment, we test the performance of the EMAC-FILTERED algorithm on a well-known benchmark problem taken from \cite{J04,MS}, known as channel flow around a cylinder and compare results with that of BE-EMAC scheme. This problem has been widely studied for simulation of fluid flows thanks to its real flow characteristics and highly reliable data to measure accuracy of methods. For the problem set-up, we follow the paper \cite{MS}. The computational domain is a $[0,2.2]\times[0,0.41]$ rectangular channel with a cylinder (circle) of radius $0.05$ centered at $(0.2,0.2)$, seen in Figure \ref{fig:d}.
\begin{figure}[h!]
	\centering
	\includegraphics[width=110mm]{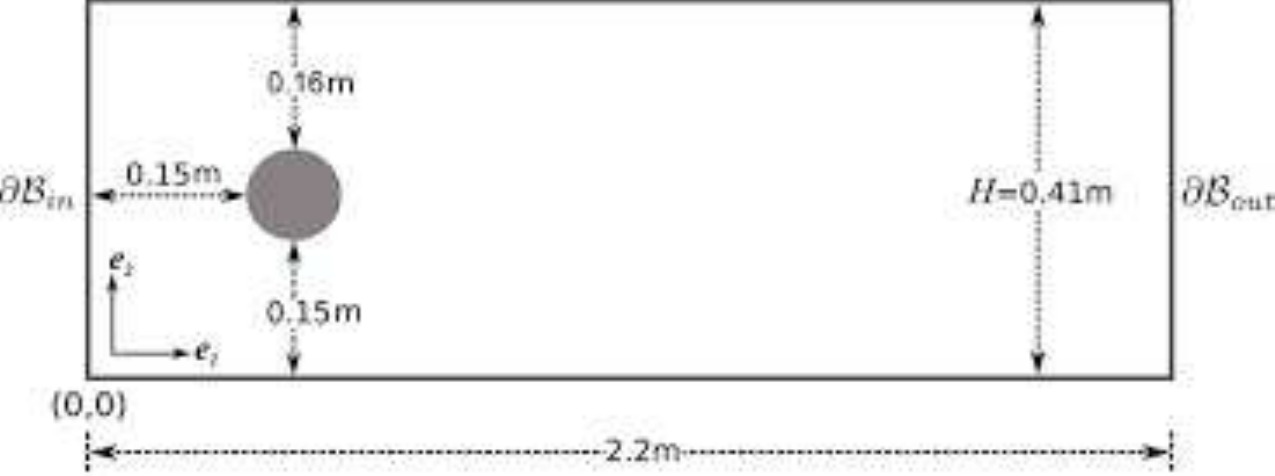}
	\caption{ Domain $\Omega$ of the test problem}
	\label{fig:d}
\end{figure}
 The time dependent inflow and outflow velocity profiles are given by
\begin{eqnarray}
\bu_1(0,y,t)&=&\bu_1(2.2,y,t)=\dfrac{6}{0.41^2}sin\big(\dfrac{\pi t}{8}\big)y(0.41-y)\nonumber\\\quad\nonumber\\\bu_2(0,y,t)&=&\bu_2(2.2,y,t)= 0\nonumber
\end{eqnarray}
No-slip boundary conditions are enforced at the cylinder and walls. We take zero initial condition $\bu(x,y,t)=0 $ and the kinematic viscosity $\nu=10^{-3}$. Moreover, there is no external force acting on the flow. We run the problem on a very coarse mesh consisting of only $ 10210$ total degrees of freedom with an end time $T=8$ and time-step $\Delta t=0.01$.

The plots of flow development of both BE-EMAC scheme and EMAC-FILTERED scheme at time $t=2,4,6,8$ are presented in Figure \ref{fig:u} and Figure \ref{fig:T}, respectively. We observe that although BE-EMAC solutions at $t=2, t=4$ are similar to the DNS of \cite{J04, MS}, solutions at $t=6, t=8$ are totally inaccurate in which even vortices are not formed which incorrectly predicts velocity solution of turbulent-like flows. However, the plots of  EMAC-FILTERED scheme matches quite well with the DNS of \cite{J04, MS} in which the formation of vortices, which are known as Von-Karman street, are  clearly observable. The results of this simulation clearly unrolls the superiority of EMAC-FILTERED scheme over unfiltered case in terms of accuracy and proves the assertion of reducing the undesirable drawbacks of BE discretization by the application of simple time filtering algorithm.

\begin{figure}[!h]
	\centering
	\begin{subfigure}[b]{1\linewidth}
		\includegraphics[width=165 mm]{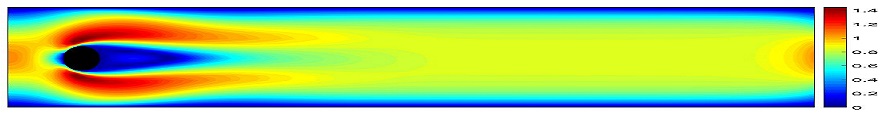}
	\end{subfigure}
	\begin{subfigure}[b]{1\linewidth}
		\includegraphics[width=165 mm]{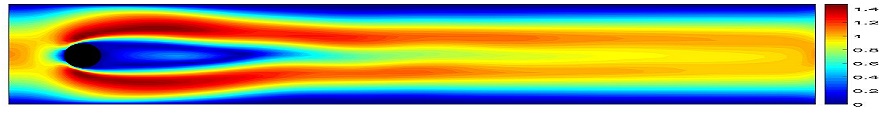}
	\end{subfigure}
	\begin{subfigure}[b]{1\linewidth}
		\includegraphics[width=165 mm]{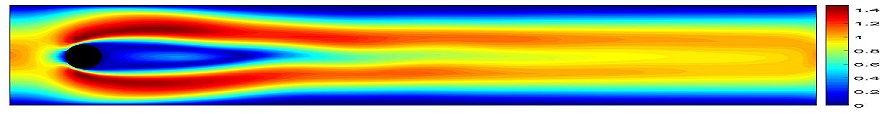}
	\end{subfigure}
	\begin{subfigure}[b]{1\linewidth}
		\includegraphics[width=165 mm]{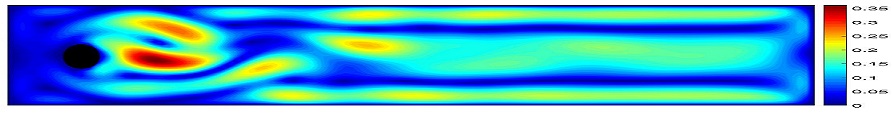}
	\end{subfigure}
	\caption{The velocity of the BE-EMAC scheme at $t = 2, 4, 6, 8$ (from up to down).}
	\label{fig:u}
\end{figure}

\begin{figure}[!h]
	\centering
	\begin{subfigure}[b]{1\linewidth}
		\includegraphics[width=165 mm]{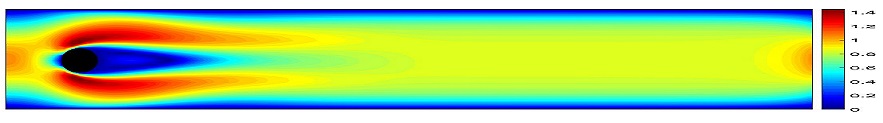}
	\end{subfigure}
	\begin{subfigure}[b]{1\linewidth}
		\includegraphics[width=165 mm]{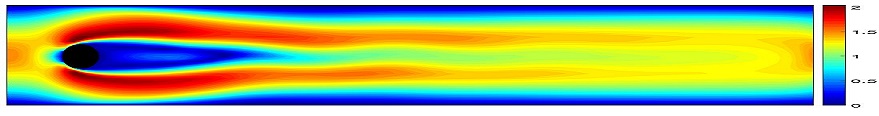}
	\end{subfigure}
	\begin{subfigure}[b]{1\linewidth}
		\includegraphics[width=165 mm]{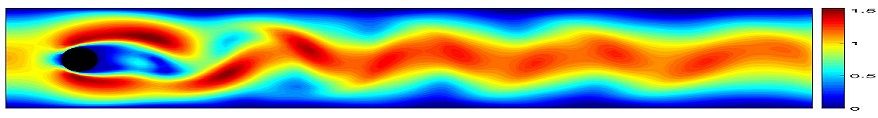}
	\end{subfigure}
	\begin{subfigure}[b]{1\linewidth}
		\includegraphics[width=165 mm]{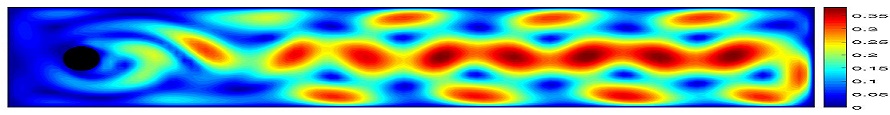}
	\end{subfigure}
	\caption{The velocity of the EMAC-FILTERED scheme at $t = 2, 4, 6, 8$ (from up to down).}
	\label{fig:T}
\end{figure}

For further observation of the accuracy of our method, we compute the drag $c_d(t)$ and lift $c_l(t)$ coefficients at the cylinder. These values are defined in \cite{MS} as follows:
\begin{eqnarray}
c_d(t)&=&\dfrac{2}{\rho L U_{max}^2}\int_{S}\big(\rho\nu\dfrac{\partial\bu_{t_S}}{\partial n}n_y-p(t)n_x\big)dS\nonumber\\\quad\nonumber\\
c_l(t)&=&-\dfrac{2}{\rho L U_{max}^2}\int_{S}\big(\rho\nu\dfrac{\partial\bu_{t_S}}{\partial n}n_x+p(t)n_y\big)dS\nonumber\\\quad\nonumber
\end{eqnarray}
where $S$ is the boundary of the cylinder, $U_{max}$ is the maximum mean flow, $L$ is the diameter of the cylinder, $n=(n_x,n_y)^T$ is the normal vector on the circular boundary $S$ and $\bu_{t_S}$ is the tangential velocity for $t_S=(n_y,-n_x)^T$ the tangential vector.\\ Table \ref{table:tab3} shows the maximum drag $(c_{d,max} )$ and maximum lift $(c_{l,max})$ values of both EMAC-FILTERED and BE-EMAC schemes behind the cylinder together with the times at which they occur. The following reference intervals are given in \cite{J04, MS}:
$$c_{d,max}^{ref}\in [2.93, 2.97],\quad \quad c_{l,max}^{ref}\in [0.47, 0.49].$$
Comparing with the reference values, we see that while both predicts maximum drag coefficients correctly, EMAC-FILTERED scheme provides the best prediction of maximum lift coefficient compared with BE-EMAC scheme which is not even in the reference interval. Thus, this numerical test has revealed that making use of EMAC-FILTERED scheme has notable advantages in terms of practical applications and quantitative means over BE-EMAC scheme.
\begin{table}[h!]
	\begin{center}
		\begin{tabular}{|c|c|c|c|c|}
			\hline
			Method & $c_{d,max}$ &$t(c_{d,max})$ & $c_{l,max}$&$t(c_{l,max})$ \\
			\hline
				EMAC-FILTERED&2.95281 & 3.94 &0.468963&5.79 \\
					\hline
			BE-EMAC&2.95215&3.93 &0.0299554&7.13  \\
			\hline
			Ref \cite{J04} &2.95092&3.93 &0.47795&5.69\\
			\hline
		
			\hline
		\end{tabular}
	\end{center}
	\caption{Comparison of maximum drag and lift coefficients and the times at which they occur.}
	\label{table:tab3}
\end{table}

 \section{Conclusion}
 In this paper, we proposed and analyzed the backward Euler based modular time filter method for the EMAC formulation of NSE. We showed that numerical accuracy is increased from first order to second order without requiring any additional computational effort. Unconditional stability and convergence results that show optimal rates were provided. A rich blend of numerical experiments verified the theoretical expectations and demonstrated reliability and efficiency of the proposed method. The numerical results clearly exposed that the EMAC-FILTERED scheme produces more accurate results and better quality solutions over the unfiltered case. Thus, the simplicity of BE discretization is combined with desirable accuracy and efficiency properties which we aim to arrive. We also showed that the EMAC-FILTERED scheme conserves important physical quantities as good, or better than the BE-EMAC scheme.
\bibliographystyle{plain}

\bibliography{references1}
\end{document}